\documentclass[12pt]{amsart}
\usepackage{amsmath,amsfonts,amssymb,amsthm}
\usepackage{anysize}
\marginsize{2cm}{2cm}{2cm}{2cm}
\usepackage{graphicx}

\usepackage{ textcomp }
\usepackage{schemata}

\usepackage{enumerate}
\usepackage{color}
\usepackage{soul}

\def\z{\mathfrak{z}}
\def\u{\mathfrak{u}}
\def\k{\mathfrak{k}}
\def\g{\mathfrak{g}}
\def\h{\mathfrak{h}}
\def\n{\mathfrak{n}}

\def\C{\mathbb{C}}
\def\R{\mathbb{R}}
\def\Q{\mathbb{Q}}
\def\Z{\mathbb{Z}}
\def\N{\mathbb{N}}

\def\ad{\operatorname{ad}}
\def\tr{\operatorname{tr}}
\def\I{\operatorname{Id}}
\def\alt{\raise1pt\hbox{$\bigwedge$}}
\def\pint{\langle \cdotp,\cdotp \rangle }
\def\la{\langle}
\def\ra{\rangle}
\def\modi{\color{red}}

\theoremstyle{plain}
\newtheorem{teo}{\bf Theorem}[section]
\newtheorem{cor}[teo]{\bf Corollary}
\newtheorem{prop}[teo]{\bf Proposition}
\newtheorem{lema}[teo]{\bf Lemma}

\theoremstyle{definition}
\newtheorem{defi}[teo]{\bf Definition}
\newtheorem{ejemplo}[teo]{\bf Example}

\theoremstyle{remark}
\newtheorem{rem}[teo]{\bf Remark}

\newcommand{\ri}{{\rm (i)}}
\newcommand{\rii}{{\rm (ii)}}

\title[Vaisman solvmanifolds]{Vaisman solvmanifolds and relations with other geometric structures}

\date{}
\author{A. Andrada}
\address[A. Andrada]{FaMAF-CIEM, Universidad Nacional de C\'{o}rdoba, Ciudad Universitaria, X5000HUA C\'{o}rdoba, Argentina}
\email{andrada@famaf.unc.edu.ar}

\author{M. Origlia}
\address[M. Origlia]{KU Leuven Kulak\\ E. Sabbelaan 53\\ BE-8500 Kortrijk, Belgium; and FaMAF-CIEM, Universidad Nacional de C\'{o}rdoba \\ X5000HUA C\'{o}rdoba\\ Argentina}
\email{origlia@famaf.unc.edu.ar}

\thanks{This work was partially supported by CONICET, SECyTUNC and ANPCyT (Argentina) and the Research Foundation Flanders (Project G.0F93.17N)}
\subjclass[2010]{22E25, 53B35, 53C25, 22E40}
\keywords{Locally conformally K\"ahler structure, Vaisman structure, solvable Lie group, lattice, solvmanifold}

\begin{document}

\begin{abstract}
We characterize unimodular solvable Lie algebras with Vaisman structures in terms of K\"ahler flat Lie algebras equipped with a suitable derivation. Using this characterization we obtain algebraic restrictions for the existence of Vaisman structures and we establish some relations with other geometric notions, such as Sasakian, coK\"ahler and left-symmetric algebra structures. Applying these results we construct families of Lie algebras and Lie groups admitting a Vaisman structure and we show the existence of lattices in some of these families, obtaining in this way many examples of new solvmanifolds equipped with invariant Vaisman structures. 
\end{abstract}

\maketitle

\section{Introduction}

Let $(M,J,g)$ be a $2n$-dimensional Hermitian manifold, where $J$ is a complex structure and $g$ is a Hermitian metric, and let $\omega$ denote its fundamental $2$-form, that is, 
$\omega(X,Y)=g(JX,Y)$ for any $X,Y$ vector fields on $M$. The manifold $(M,J,g)$ is called {\it locally conformally K\"ahler} (LCK) if $g$ can be rescaled locally, in a 
neighborhood of any point in $M$, so as to be K\"ahler, i.e., there exists an open covering $\{ U_i\}_{i\in I}$ of $M$ and a family $\{ f_i\}_{i\in I}$ of $C^{\infty}$ functions, 
$f_i:U_i \to \R$, such that each local metric 
\begin{equation}\label{gi} 
g_i=\exp(-f_i)\,g|_{U_i} 
\end{equation} 
is K\"ahler. These manifolds are a natural generalization of the class of K\"ahler manifolds, and they have been much studied by many authors since the work of I. Vaisman in the 
'70s (see for instance \cite{DO, GMO, O,V4}).

An equivalent characterization of an LCK manifold can be given in terms of the fundamental form $\omega$. Indeed, a Hermitian manifold $(M,J,g)$ is LCK if and only if there exists a closed $1$-form $\theta$ globally defined on $M$ such that 
\begin{equation}\label{lck}
d\omega=\theta\wedge\omega.
\end{equation} 
This closed $1$-form $\theta$ is called the \textit{Lee form} (see \cite{L}). Furthermore, the Lee form $\theta$ is uniquely determined by the following formula: 
\begin{equation}\label{tita} 
\theta=-\frac{1}{n-1}(\delta\omega)\circ J, 
\end{equation} 
where $\omega$ is the fundamental $2$-form, $\delta$ is the codifferential operator and $2n$ is the dimension of $M$. A Hermitian $(M,J,g)$ is called \textit{globally conformally 
K\"ahler} (GCK) if there exists a $C^{\infty}$ function, $f:M\to\R$, such that the metric $\exp(-f)g$ is K\"ahler, or equivalently, the Lee form is exact, $\theta=df$. Therefore a simply 
connected LCK 
manifold is GCK. 

It is well known that LCK manifolds belong to the class $\mathcal{W}_4$ of the Gray-Hervella classification of almost Hermitian manifolds \cite{GH}. Also, an LCK manifold 
$(M,J,g)$ with $\dim M\geq 4$ is K\"ahler if and only if $\theta=0$. Indeed, $\theta\wedge\omega=0$ and $\omega$ non degenerate imply $\theta=0$. It is known that if $(M,J,g)$ is a Hermitian manifold 
with $\dim M\ge 6$ such that \eqref{lck} holds for some $1$-form $\theta$, then $\theta$ is automatically closed, and 
therefore $M$ is LCK.

The Hopf manifolds are examples of LCK manifolds, and they are obtained as a quotient of $\C^n-\{0\}$ with the Boothby metric by a discrete subgroup of automorphisms. These 
manifolds are diffeomorphic to $S^1\times S^{2n-1}$ and have first Betti number $b_1$ equal to 1, so that they do not admit any K\"ahler metric. The LCK structures on these Hopf 
manifolds have a special property, as shown by Vaisman in \cite{V2}. Indeed, the Lee form is parallel with respect to the Levi-Civita connection of the Hermitian metric. The 
LCK manifolds sharing this property form a distinguished class, which has been much studied since Vaisman's seminal work \cite{GMO,KS,OV3,OV4,V2,V3}. 

\medskip

\begin{defi}
$(M,J,g)$ is a Vaisman manifold if it is LCK and the Lee form $\theta$ is parallel with respect to the Levi-Civita connection.
\end{defi}

A Vaisman manifold satisfies stronger topological properties than general LCK manifolds. For instance, a compact Vaisman non-K\"ahler manifold $(M,J,g)$ has 
$b_1(M)$ odd (\cite{KS,V3}), whereas in \cite{OT} an example is given of a compact LCK manifold with even $b_1(M)$. This also implies that a compact Vaisman manifold cannot admit 
K\"ahler metrics, since the odd Betti numbers of a compact K\"ahler manifold are even. Moreover, it was proved in \cite[Structure Theorem]{OV3} and \cite[Corollary 3.5]{OV4} that 
any compact Vaisman manifold admits a Riemannian submersion to a circle such that all fibers are isometric and admit a natural Sasakian structure. It was shown in \cite {Ve} 
that any compact complex submanifold of a Vaisman manifold is Vaisman, as well. In \cite{Bel} the classification of compact complex surfaces admitting a Vaisman structure is 
given. It is known that a homogeneous LCK manifold is Vaisman when the manifold is compact (\cite{GMO,HK2}) and, more generally, when the manifold is a quotient of a reductive Lie 
group such that the normalizer of the isotropy group is compact (\cite{ACHK}).

\

In this article we are interested in invariant Vaisman structures on solvmanifolds, that is, compact quotients $\Gamma\backslash G$ where $G$ is a simply connected solvable Lie 
group and $\Gamma$ is a lattice in $G$. We begin by studying left invariant Vaisman structures on Lie groups, or equivalently, Vaisman structures on a Lie algebra.

Let $G$ be a Lie group with a left invariant complex structure $J$ and a left invariant metric $g$. If $(G,J,g)$ satisfies the LCK condition \eqref{lck}, then $(J,g)$ is called a 
{\em left invariant LCK structure} on the Lie group $G$. In this case, it follows from \eqref{tita} that the corresponding Lee form $\theta$ on $G$ is also 
left invariant. 

This fact allows us to define LCK structures on Lie algebras. Recall that a \textit{complex structure J} on a Lie algebra $\g$ is an endomorphism $J: \g \to \g$ satisfying 
$J^2=-\I$ and 
\[ N_J=0, \quad  \text{where} \quad N_J(x,y)=[Jx,Jy]-[x,y]-J([Jx,y]+[x,Jy]),\]        
for any $x,y \in \g$. 

Let $\g$ be a Lie algebra, $J$ a complex structure and $\pint$ a Hermitian inner product on $\g$, with $\omega\in\alt^2\g^*$ the fundamental $2$-form. We say that $(\g,J,\pint)$ 
is \textit{locally conformally K\"ahler} (LCK) if there exists $\theta \in \g^*$, with $d\theta=0$, such that 
\begin{equation} \label{g-lck-0}
d\omega=\theta\wedge\omega.
\end{equation}
Here $d$ denotes the coboundary operator of the Chevalley-Eilenberg complex of $\g$ corresponding to the trivial representation. 

If the Lie group $G$ is simply connected then any left invariant Vaisman structure on $G$ turns out to be globally conformal to a K\"ahler structure. Therefore we will study 
compact quotients of such a Lie group by discrete subgroups (if they exist); these quotients will be non simply connected and will inherit a Vaisman structure. Recall that a 
discrete subgroup $\Gamma$ of a simply connected Lie group $G$ is called a \textit{lattice} if the quotient $\Gamma\backslash G$ is compact. According to \cite{Mi}, if such a 
lattice exists then the Lie group must be unimodular. The quotient $\Gamma\backslash G$ is known as a solvmanifold if $G$ is solvable and as a nilmanifold if $G$ is nilpotent, and 
in these cases we have that $\pi_1(\Gamma\backslash G)\cong \Gamma$. Moreover, the diffeomorphism class of solvmanifolds is determined by the isomorphism class of the 
corresponding lattices, as the following results show:

\begin{teo}\cite[Theorem 3.6]{R}
Let $G_1$ and $G_2$ be simply connected solvable Lie groups and $\Gamma_i$, $i=1,2$, a lattice in $G_i$. If $f:\Gamma_1\to\Gamma_2$ is an isomorphism, then there exists a 
diffeomorphism $F:G_1\to G_2$ such that
\begin{enumerate}[(i)]
 \item $F|_{\Gamma_1}=f$,
 \item $F(\gamma g)=f(\gamma)F(g)$, for any $\gamma\in \Gamma_1$ and $g\in G_1$.
\end{enumerate}
\end{teo}

\begin{cor}\cite{Mo} \label{mostow}
 Two solvmanifolds with isomorphic fundamental groups are diffeomorphic.
\end{cor}

\

LCK and Vaisman structures on Lie groups and Lie algebras and also on their compact quotients by discrete subgroups have been studied by several authors lately 
(see \cite{ACHK,ACFM,Baz,CFL,HK, HK2,K,S,S1,S2,S3} among others). For instance, it was shown in \cite{S} that if an LCK Lie algebra is nilpotent then it is isomorphic to 
$\h_{2n+1}\times\R$ and the LCK structure is Vaisman (see Example \ref{heisenberg}, cf. \cite{HK}). In \cite{ACHK} the authors prove that if a reductive Lie algebra admits an LCK structure then it is isomorphic to either 
$\u(2)$ or $\mathfrak{gl}(2,\R)$. In \cite{K} it is proved the non-existence of Vaisman metrics on some solvmanifolds with left invariant complex structures. In \cite{Baz} it is 
proved that if a nilmanifold $\Gamma\backslash G$ admits a Vaisman structure (not necessarily invariant), then $G$ is isomorphic to the cartesian product of a Heisenberg group 
$H_{2n+1}$ with $\R$. In \cite{S3} it is shown that if a completely solvable solvmanifold equipped with an invariant complex structure admits a Vaisman metric, then the 
solvmanifold is again a quotient of $H_{2n+1}\times \R$. In \cite{MP} the authors obtain a Vaisman structure on the total space of certain $S^1$-bundles over compact coK\"ahler 
manifolds, and all the examples they exhibit are diffeomorphic to compact solvmanifolds. Recently, in the preprint \cite{AHK}, it was shown that any unimodular Vaisman Lie algebra is isomorphic, up to modifications, to one of the following Lie algebras: $\h_{2n+1}\times\R$, $\mathfrak{su}(2)\times \R$ or $\mathfrak{sl}(2,\R)\times\R$ (see \cite{AHK} for the relevant definitions).

\medskip

In this article we obtain a characterization of the unimodular solvable Lie algebras admitting Vaisman structures in terms of K\"ahler flat Lie algebras equipped with suitable derivations (see Theorems \ref{Main Theorem-A} and \ref{Main Theorem-B}). More explicitly, we show that any unimodular 
solvable Vaisman Lie algebra is a double extension of a K\"ahler flat Lie algebra. In order to do this, we use the fact that Vaisman structures are closely related to Sasakian structures. 

This characterization allows us to build new examples of unimodular solvable non-nilpotent Lie algebras equipped with Vaisman structures. When these Lie algebras have 
integer structural constants we exhibit families of lattices in the associated simply connected solvable Lie groups, and we show that the solvmanifolds obtained in this way are 
not diffeomorphic to the product of $S^1$ with a Heisenberg nilmanifold.

Moreover, we establish a relation with other geometric structures, namely, with coK\"ahler Lie algebras and left-symmetric algebras. More precisely, we show that 
any unimodular solvable Vaisman Lie algebra is a central extension of a coK\"ahler flat Lie algebra, and using this we prove the existence of a complete left-symmetric algebra 
structure on the Vaisman Lie algebra. This gives rise to a complete flat torsion-free connection on any associated solvmanifold. 

\medskip

The article is organized as follows. In \S 2 we prove a general result about unimodular LCK Lie algebras and we recall some basic definitions. In \S 3 we review some 
properties about Vaisman Lie algebras and we give the proof of the main theorems (Theorems \ref{Main Theorem-A} and \ref{Main Theorem-B}). As a consequence of these theorems, we 
need to study derivations of a K\"ahler Lie algebra, and we do this in \S 4. In \S 5, we obtain a strong restriction for the existence of Vaisman structures, namely, if a 
unimodular solvable Lie algebra admits such a structure then the spectrum of $\ad_X$ is contained in $i\R$ for any $X$ in the Lie algebra (see Theorem \ref{Vaisman-imag.puras}). 
In \S 6, using the characterization obtained previously we provide families of new examples of unimodular solvable Vaisman Lie algebras in any even dimension and determine the 
existence of lattices in many of the corresponding solvable Lie groups. Finally, in \S 7 we prove the relation mentioned above with coK\"ahler Lie algebras (Theorem 
\ref{cokahler}) and left-symmetric algebras (Corollary \ref{LSA2}).

\medskip

\textbf{Acknowledgments.} The authors are grateful to I. Dotti and K. Hasegawa for their useful comments and to the referees for their careful reading of the manuscript and their suggestions. 

\

\section{Preliminaries}\label{prelim}

Let $(\g,J,\pint)$ be a Lie algebra with an LCK structure. We have the following orthogonal decomposition for $\g$, 
\[\g=\mathbb{R}A \oplus \ker\theta\] 
where $\theta$ is the Lee form and $\theta(A)=1$. Since $d\theta=0$, we have that $\g'=[\g,\g]\subset\ker\theta$. It is clear that $JA\in\ker\theta$, but when $\g$ is unimodular we 
may state a stronger result. Recall that a Lie algebra is unimodular if $\tr(\ad_x)=0$ for all $x$ in the Lie algebra.

\begin{prop}\label{JA-conm}
If $\g$ is unimodular and $(J,\pint)$ is an LCK structure on $\g$, then $JA\in\g'$.
\end{prop}

\begin{proof}
Let $\{e_1,\dots, e_{2n}\}$ be an orthonormal basis of $\g$. Recall from \cite{Be} the following formula  
\[\delta\eta=-\sum_{i=1}^{2n} \iota_{e_i} (\nabla_{e_i} \eta),\]
where $\eta$ is a $p$-form and $\delta$ is the codifferential operator. Note that the Koszul formula for the Levi-Civita connection in this setting is given simply by 
\[ \la \nabla_xy,z\rangle= \frac12 \left(\la[x,y],z\ra-\la[y,z],x\ra+\la [z,x],y\ra\right), \qquad x,y,z\in\g.\]
Using this formula we compute $\delta\omega$, where $\omega$ is the fundamental $2$-form. For $x\in\g$,  
\begin{align*}
\delta\omega(x) = & -\sum_i (\nabla_{e_i} \omega)(e_i,x) \\
  = & \sum_i \omega(\nabla_{e_i}e_i,x)+\omega(e_i,\nabla_{e_i}x) \\
  = & \sum_i -\langle\nabla_{e_i}e_i,Jx\rangle+\langle Je_i,\nabla_{e_i}x\rangle \\
  = & \frac12\left\{\sum_i \langle[e_i,Jx],e_i\rangle-\langle[Jx,e_i],e_i\rangle+\langle[e_i,x],Je_i\rangle-\langle[x,Je_i],e_i\rangle +\langle[Je_i,e_i],x\rangle\right\} \\
  = & \frac12 \left\{-2\tr(\ad_{Jx})+\tr(J\circ \ad_x)-\tr(\ad_x\circ J)+ \sum_i\langle[Je_i,e_i],x\rangle\right\} \\
 = & \frac12\sum_i\langle[Je_i,e_i],x\rangle{\modi{.}}
\end{align*} 
It follows from \eqref{tita} that $\theta (x)=\frac{1}{2(n-1)}\sum\langle J[Je_i,e_i],x\rangle$. On the other hand, the Lee form can be written in terms of the inner product as 
$\theta(x)=\frac{\la A,x\ra}{|A|^2}$. If we compare both expressions we obtain that \[A=\frac{|A|^2}{2(n-1)}\sum_i J[Je_i,e_i].\] Therefore $JA\in\g'$.
\end{proof}

\

We will see in forthcoming sections that Vaisman structures on Lie algebras are closely related to certain almost contact metric structures on lower-dimensional Lie algebras. 
Moreover, when the Vaisman Lie algebra is unimodular and solvable  we will show that it is a double extension of a K\"ahler flat Lie algebra. Let us recall the relevant 
definitions.

\medskip

\subsection{Almost contact metric Lie algebras} 
An almost contact metric structure on a Lie algebra $\h$ is a quadruple $(\pint, \phi, \xi, \eta)$, where $\pint$ is an inner 
product on $\h$, $\phi$ is an endomorphism 
$\phi:\h\to\h$, and $\xi\in\h$, $\eta\in\h^*$ satisfy the following conditions:
\begin{itemize}\label{sasakian}
\item $\eta(\xi)=1$,
\item $\phi^2=-\I+\eta\otimes\xi$,
\item $\langle\phi x,\phi y\rangle=\langle x,y\rangle -\eta(x)\eta(y)$, for all $x,y\in\h$.
\end{itemize}
It follows that $|\xi|=1$, $\phi(\xi)=0$, $\eta\circ\phi=0$, and $\phi$ is skew-symmetric. The fundamental 2-form $\Phi$ associated to $(\pint, \phi, \xi, \eta)$ is defined by 
$\Phi(x,y)=\langle \phi x, y\rangle$, for $x,y\in\h$. The almost contact metric structure is called:
\begin{itemize}
 \item \textit{normal} if $N_\phi=-d\eta\otimes \xi$;
 \item \textit{Sasakian} if it is normal and $d\eta=2\Phi$;
 \item \textit{almost coK\"ahler} if $d\eta=d\Phi=0$;
 \item \textit{coK\"ahler} if it is almost coK\"ahler and normal (hence $N_\phi=0$). Equivalently, $\phi$ is parallel (see \cite{Bl}).
\end{itemize}
Here $N_\phi$ denotes the Nijenhuis tensor associated to $\phi$, which is defined, for $x,y\in\h$, by 
\begin{equation}\label{nijen}
N_\phi(x,y) = [\phi x,\phi y] +\phi^2[x,y] -\phi([\phi x,y] +[x,\phi y]). 
\end{equation}

\smallskip

CoK\"ahler structures are also known as ``cosymplectic'', following the terminology introduced by Blair in \cite{Bl1} and used in many articles since then, but their 
striking analogies with K\"ahler manifolds have led Li (see \cite{Li}) and other authors to use the term ``coK\"ahler'' for these structures, and this is becoming common practice. 
In the present article we follow this terminology.

\medskip

\begin{rem}
Let $\h$ be a Lie algebra equipped with a Sasakian structure $(\pint, \phi, \eta, \xi)$. It follows from $d\eta=2\Phi$ that $\eta$ is a contact form on $\h$, and consequently 
$\xi$ 
is called the {\em Reeb} vector. It is easy to verify that the center of $\h$ has dimension at most $1$. Moreover, if $\dim\z(\h)=1$ then the center is generated by the Reeb 
vector 
(see \cite{AFV}).  
\end{rem}

\smallskip

We recall next a result about Sasakian Lie algebras, which will be necessary to 
prove our main results.

\begin{prop}[\cite{AFV}]\label{Sasakiana}
Let $(\phi,\eta,\xi,\pint)$ be a Sasakian structure on a Lie algebra $\mathfrak{h}$ with non trivial center $\z(\mathfrak{h})$ generated by $\xi$. If $\k:=\ker \eta$, then the 
quadruple $(\k,[\cdot,\cdot]_{\k},\phi|_{\k},\pint|_{\k\times\k})$ is a K\"ahler Lie algebra, 
where $[\cdot,\cdot]_{\k}$ is the component of the Lie bracket of $\h$ on $\k$. 
\end{prop}

\

\subsection{Double extension of Lie algebras}

Let $\h$ be a real Lie algebra and $\beta\in\alt^2 \h^*$ a closed $2$-form. If we consider $\R$ as the $1$-dimensional abelian Lie algebra, generated by an element $\xi\in\R$, we 
may define on the vector space $\R\xi\oplus \h$ the following bracket:
\[ [x,y]_\beta=\beta(x,y)\xi + [x,y]_\h, \qquad [\xi,x]_\beta=0, \qquad x,y\in\h. \]
It is readily verified that this bracket satisfies the Jacobi identity, and $\R\xi\oplus \h$ will be called the central extension of $\h$ by the closed $2$-form $\beta$. It will 
be denoted $\h_\beta(\xi)$.

Given a derivation $D$ of $\h_\beta(\xi)$, the double extension of $\h$ by the pair $(D,\beta)$ is defined as the semidirect product $\h(D,\beta):=\R \ltimes_D \h_\beta(\xi)$ (see 
\cite{ARS} for more details). 

\medskip

\begin{lema}\label{unimod}
 Let $\h(D,\beta)$ be the double extension of $\h$ by the pair $(D,\beta)$. Then $\h(D,\beta)$ is unimodular if and only if $\h$ is unimodular and $\tr D=0$.
\end{lema}

\begin{proof}
Let us denote $\g=\h(D,\beta)$. 
Let $A$ be a generator of $\R$, so that $[A,x]=Dx$ for all $x\in \h_\beta(\xi)$.

Fix any inner product $\pint$ on $\g$ such that $\text{span}\{A,\xi\}$ is orthogonal to $\h$, $\langle A,\xi\rangle=0$ and $|A|=|\xi|=1$. Given an orthonormal basis 
$\{e_1,\dots,e_{n}\}$ of $\h$,  we have that $ \{A, \xi\}\cup\{e_1,\dots,e_n\}$ is an orthonormal basis of $\g$.
For any $x\in\h_\beta(\xi)$, we compute 
\begin{align*}
\tr(\ad_x^\g) =& \langle [x,A]_\g,A\rangle +\langle [x,\xi]_\beta,\xi\rangle + \sum_{i=1}^n \langle [x,e_i]_\beta,e_i\rangle  \\
=& \sum_{i=1}^n \left( \langle [x,e_i]_\h,e_i\rangle + \langle \beta(x,e_i)\xi,e_i\rangle \right) \\
=& \tr(\ad_x^\h).
\end{align*}
From this and the fact that $\tr(\ad^\g_A)=\tr D$, the result follows.
\end{proof}

\

\section{Vaisman structures on Lie algebras}

In this section we show the main results of this article, namely, a characterization of unimodular solvable Lie algebras admitting a Vaisman structure (Theorems \ref{Main 
Theorem-A} and \ref{Main Theorem-B}). In order to prove them, we establish first basic properties of Vaisman Lie algebras and later we exploit the close relation between Vaisman 
and Sasakian structures. 

\medskip

A \textit{Vaisman structure} on a Lie algebra $\g$ is an LCK structure $(J,\pint)$ such that the associated Lee form $\theta$ satisfies $\nabla \theta=0$. The Lie algebra $\g$ 
together with the Vaisman structure $(J,\pint)$ will be called a Vaisman Lie algebra.

\smallskip

If $\g=\R A\oplus \ker\theta$ with $A\in(\ker\theta)^\perp$ such that $\theta(A)=1$,  then the Vaisman condition is equivalent to $\nabla A=0$, and 
since $d\theta=0$, this is in turn equivalent to $A$ being a Killing vector field (considered as a left invariant vector field on the associated Lie group with left invariant 
metric). Recalling that a left invariant vector field is Killing if and only if the corresponding adjoint operator on the Lie algebra is skew-symmetric, we have:

\begin{prop}[\cite{AO}]\label{adA-antisim}
If $(J,\pint)$ is an LCK structure on $\g$, then it is Vaisman if and only if the endomorphism $\ad_A$ is skew-symmetric.
\end{prop}

\smallskip

\begin{rem}
In \cite{S2} it was proved that an LCK structure on a unimodular solvable Lie algebra is Vaisman if and only if $\langle [A,JA],JA\rangle =0$, where $A$ is the metric dual of 
$\theta$. However, the characterization given in Proposition \ref{adA-antisim} will be more useful for our purposes. 
\end{rem}

\medskip

\begin{ejemplo}\label{heisenberg}
Let $\g=\R\times\h_{2n+1}$, where $\h_{2n+1}$ is the $(2n+1)$-dimensional
Heisenberg Lie algebra. There is a basis $\{x_1,\dots,x_n,y_1,\dots,y_n,z,w\}$ of $\g$ with Lie
brackets given by $[x_i,y_i]=z$ for $i=1,\dots,n$ and $w$ in the center. We define an inner
product $\pint$ on $\g$ such that the basis above is orthonormal. Let $J$ be the almost complex structure on $\g$ given by:
\[Jx_i=y_i, \quad  Jz=-w  \; \; \; \text{for $i=1,\dots,n$}.\]
It is easily seen that $J$ is a complex structure on $\g$ compatible with $\pint$. If $\{x^i,y^i,z^*,w^*\}$ denote the $1$-forms dual to $\{x_i,y_i,z,w\}$ respectively, then the 
fundamental $2$-form is: 
\[\omega=\sum_{i=1}^n(x^i\wedge y^i) - z^*\wedge w^*.\]
Thus, \[ d\omega=w^*\wedge\omega,\]
and therefore $(\g,J,\pint)$ is LCK. It follows from Proposition \ref{adA-antisim} with $A=w$ that this structure is Vaisman. This example appeared in
\cite{CFL} (see also \cite{S,AO}).

It is known that $\g$ is the Lie algebra of the Lie group $\R\times H_{2n+1}$,
where $H_{2n+1}$ is the $(2n+1)$-dimensional Heisenberg group. The Lie group $H_{2n+1}$ admits a lattice $\Gamma$ and therefore the nilmanifold $N= S^1 \times
\Gamma\backslash H_{2n+1}$ admits an LCK structure which is Vaisman. It cannot admit any K\"ahler metric (see \cite{BG}), and for $n=1$ the nilmanifold $N$ is a primary Kodaira surface.
\end{ejemplo}

\

The following important properties of Vaisman Lie algebras follow from Proposition \ref{adA-antisim} and the integrability of the complex structure (see also \cite{V2}):

\begin{prop}\label{vaisman-properties}
Let $(\g,J,\pint)$ be a Vaisman Lie algebra, then
\begin{enumerate}
\item $[A,JA]=0$,
\item $J\circ\ad_A=\ad_A\circ J$,
\item $J\circ\ad_{JA}=\ad_{JA}\circ J$,
\item $\ad_{JA}$ is skew-symmetric.
\end{enumerate}
\end{prop}

\

Without loss of generality, we will assume from now on that $|A|=1$ (rescaling the metric if necessary), hence $\theta(x)=\langle A, x \rangle$ for all $x\in\g$. Let us denote $W=(\text{span}\{A,JA\})^\perp$, so that $\ker\theta=\R JA\oplus^\perp W$.

\medskip 

The following proposition, which shows the close relation between Vaisman and Sasakian structures, follows from general results proved by I. Vaisman, but we include a proof at the 
Lie algebra level for the sake of completeness. We will use the following convention for the action of a complex structure on a 1-form: if $\alpha$ is a 1-form, then 
$J\alpha:=-\alpha\circ J$.  

\begin{prop}\label{casi_sasakiana}
Set $\xi:=JA$,  $\eta:=J\theta|_{\ker\theta}$, and define an endomorphism $\phi\in\operatorname{End}(\ker
\theta)$ by $\phi(a\xi+x)=Jx$ for $a\in\R$ and $x\in W$. Then the following relations hold:
\begin{enumerate}
\item $\phi^2=-\I+\eta\otimes\xi$,
\item $\langle\phi x,\phi y\rangle=\langle x,y\rangle -\eta(x)\eta(y)$, for all $x,y\in\ker\theta$,
\item $N_\phi=-d\eta\otimes\xi$,
\item $d\eta(x,y)=-\langle\phi x, y \rangle$, for all $x,y\in\ker\theta$,
\end{enumerate}
where $N_\phi$ is defined as in \eqref{nijen}.
\end{prop}

\begin{proof}
Note that $W=\ker\eta$ and  $\eta(\xi)=1$.

$(1)$ and $(2)$ follow from the fact that $(J,\pint)$ is a Hermitian structure on $\g$ and $|A|=1$. Since $\g$ is Vaisman, using Proposition \ref{adA-antisim} and Proposition 
\ref{vaisman-properties}, it can be seen that 
\[N_\phi(x,y)=N_J(x,y)- d\eta(x,y)\xi,\] 
for all $x,y\in\ker\theta$. Since $J$ is integrable, we have that $N_J=0$ and then we obtain $(3)$.

In order to prove $(4)$ we compute, for all $x,y\in\ker\theta$, 
\[d\eta(x,y) = \theta(J[x,y])=\la A, J[x,y] \ra=\omega([x,y],A).\] 
On the other hand,
\begin{align*}
\la Jx,y\ra &= \omega(x,y)\\
&= \theta\wedge\omega(x,y,A)\\
&= d\omega(x,y,A)\\
&= -\omega([x,y],A)-\omega([y,A],x)-\omega([A,x],y)\\
&= -\omega([x,y],A),
\end{align*}
where we have used Proposition \ref{vaisman-properties} in the last step. It is easy to verify that $\la Jx,y\ra=\la \phi x, y\ra$ for any $x,y\in\ker\theta$, thus the proof is 
complete.
\end{proof}

\smallskip

The quadruple $(\pint|_{\ker\theta},\phi,\eta,\xi)$ on $\ker\theta$ from Proposition \ref{casi_sasakiana} does not satisfy exactly the equations of a Sasakian structure given in 
\S 
\ref{prelim}, but it is easy to show that if we modify it as follows:
\[\pint'=\frac14\pint, \quad \phi'=\phi, \quad  \eta'=-\frac12 \eta, \quad \xi'=-2\xi,\]
then $(\pint',\phi',\eta',\xi')$ is a Sasakian structure on $\ker\theta$. However, in this article, for simplicity, we shall call $(\pint|_{\ker\theta},\phi,\eta,\xi)$ a Sasakian 
structure on $\ker\theta$. More generally, when we refer to a Sasakian structure on a Lie algebra we will be assuming that it satisfies the equations on Proposition 
\ref{casi_sasakiana}. Therefore, we may rewrite Proposition \ref{casi_sasakiana} as 

\begin{cor}\label{kernel}
If $(\g,J,\pint)$ is a Vaisman Lie algebra with Lee form $\theta$, then $\ker\theta$ has a Sasakian structure.
\end{cor}

\medskip

Conversely, let $\h$ be a Lie algebra equipped with a Sasakian structure $(\pint, \phi, \eta, \xi)$. Taking into account Propositions \ref{adA-antisim} and 
\ref{vaisman-properties} 
we define the Lie algebra $\g=\R A\ltimes_D\h$ where $D$ is a skew-symmetric derivation of $\h$ such that $D(\xi)=0$ and $D\phi=\phi D$ on $\ker\eta$. We consider 
on $\g$ the almost complex structure $J$ given by $J|_{\ker\eta}:=\phi|_{\ker\eta}$, $JA=\xi$, and we extend $\pint$ to an inner product on $\g$ such that $A$ is orthogonal to 
$\h$ and $|A|=1$. Note that $(J,\pint)$ is an almost hermitian structure on $\g$. It is easy to prove that $(J,\pint)$ is in fact a Vaisman structure on $\g$.

\

From now on, we assume that $\g$ is solvable and unimodular (this is a necessary condition for the associated simply connected Lie group to admit lattices, according to 
\cite{Mi}). 
The next step in order to characterize the Lie algebras admitting 
Vaisman structures is to prove that $JA$ is a central element of $\g$. Moreover, the dimension of $\z(\g)$, the center of $\g$, is at most 2.

\begin{teo}\label{JA-central}
Let $\g$ be a unimodular solvable Lie algebra equipped with a Vaisman structure $(J,\pint)$. Then $JA\in\mathfrak{z}(\g)$. Moreover, $\mathfrak{z}(\g)\subset\text{span}\{A,JA\}$.
\end{teo}

\begin{proof}
It follows from Proposition \ref{JA-conm} that $JA\in\g'$. As $\g$ is solvable, it follows that $\g'$ is nilpotent and hence $\ad_{JA}:\g\to\g$ is a nilpotent endomorphism. On the 
other hand, we know that $\ad_{JA}$ is skew-symmetric, according to Proposition \ref{vaisman-properties}, and therefore $\ad_{JA}=0$, that is, $JA\in\z(\g)$.

Now we will see that $\z(\g)\subset\{A,JA\}$. For $z\in\z(\g)$, we may assume that $z=aA+z'$ with $a\in\R$ and $z'\in W$, since $JA\in\z(\g)$. 
We have that $0=\ad_z=a\ad_A+\ad_{z'}$, and it follows from Proposition \ref{adA-antisim} that $\ad_{z'}$ is a skew-symmetric endomorphism of $\g$.
If $[z',Jz']=cJA + u$ for some $c\in\R, u\in W$, then $c=\la [z',Jz'], JA\ra=-\la Jz', [z',JA]\ra=0$.
On the other hand, taking into account  Proposition \ref{casi_sasakiana}, we have that 
\begin{align*}
c &=\la [z',Jz'], JA\ra =\eta([z',Jz'])\\
  &= -d\eta(z',Jz') = \la Jz',Jz'\ra \\
  &=|z'|^2,
\end{align*}
Therefore $z'=0$, and then $z=aA$. 
\end{proof}

\smallskip

\begin{cor}\label{sasaki-centro}
Any unimodular solvable Lie algebra admitting a Sasakian structure has non trivial center.
\end{cor}

\begin{proof}
Let $\h$ be a unimodular solvable Lie algebra equipped with a Sasakian structure. Therefore, the semidirect product $\g=\R A\ltimes_D \h$, where $D$ is a suitable skew-symmetric derivation of 
$\h$, admits a Vaisman structure. It is clear that $\g$ is unimodular and solvable. It follows from Theorem \ref{JA-central} that $\h=\ker\theta$ has non trivial 
center.
\end{proof}

\

According to Corollary \ref{sasaki-centro} the Sasakian Lie algebra $\ker\theta$ has non trivial center, which is therefore generated by the Reeb vector $JA$, and $\ker\theta$ can 
be decomposed orthogonally as
\[ \ker\theta=\mathbb{R}JA \oplus W. \] 
For $x,y\in W $, it is easy to show that
\begin{equation}\label{corchetek}
[x,y]=\omega(x,y)JA + [x,y]_W,
\end{equation}
where $[x,y]_W \in W$ denotes the component of $[x,y]$ in $W$.
It follows from Proposition \ref{Sasakiana} that $(W,[\cdot,\cdot]_W,J|_W,\pint|_{W\times W})$ is a K\"ahler Lie algebra. We will denote by $\k$ the Lie algebra
$(W,[\cdot,\cdot]_W)$. We point out that $\k$ is not a Lie subalgebra from either $\ker\theta$ or $\g$, however, it is clear from \eqref{corchetek} that $\ker\theta$ is the 
central extension $\ker\theta=\k_{\omega'}(JA)$, where  $\omega'=\omega|_{\k\times\k}$ is the fundamental 2-form on $\k$. Note that $\omega'$ is closed since $(\k,J|_{\k\times\k},\pint|_{\k\times \k} )$ is K\"ahler. Moreover, if we 
denote $D:=\ad_A|_{\ker\theta}$, then $D$ is a skew-symmetric derivation of $\ker\theta$, according to Proposition \ref{adA-antisim}. Therefore, $\g$ can be decomposed 
orthogonally as 
\begin{equation}\label{extension_doble}
 \g=\mathbb{R}A\ltimes_D(\mathbb{R}JA \oplus_{\omega'}\k),
\end{equation}
thus $\g$ can be regarded as a double extension $\g=\k(D,\omega')$. Note that, according to Proposition \ref{vaisman-properties}, $D(JA)=0$.

\medskip

According to Lemma \ref{unimod}, $\k$ is unimodular, therefore $\k$ is a unimodular Lie algebra admitting a K\"ahler structure $(J|_\k,\pint|_\k)$. Due to a classical result from 
Hano \cite{Hano}, the metric $\pint|_\k$ is \textit{flat}. 

\medskip

This leads us to use the following terminology: A pair $(\k,\pint)$ of a Lie algebra $\k$ equipped with a flat metric $\pint$ will be called a flat Lie algebra. By abuse of 
notation, we will say simply sometimes that $\k$ is a \textit{flat Lie algebra}. Furthermore, if $(J,\pint)$ is a K\"ahler structure on a Lie algebra $\k$ such that $\pint$ is 
flat, then $(\k,J,\pint)$ (or simply $\k$) will be called a \textit{K\"ahler flat Lie algebra}. 

\medskip

Resuming our discussion on unimodular solvable Vaisman Lie algebras, we obtain that $\g$ is a double extension of a K\"ahler flat Lie algebra.

Let us denote $D':=D|_{\mathfrak k}$. Then it follows from Proposition \ref{vaisman-properties} that $D'$ commutes with $J|_{\mathfrak k}$ and therefore 
$D'\in\mathfrak{u}(\k,J|_{\k}, \pint|_\k)$. Furthermore, $D'$ is a derivation of $\k$. In fact, given $x,y\in\k$ we have that
\begin{align*} 
D'[x,y]_\k & = D([x,y]-\omega(x,y)JA)\\
    & =D[x,y]\\
    & =[Dx,y]+[x,Dy]\\
    & =[D'x,y]+[x,D'y]\\
    & =\omega(D'x,y)JA+[D'x,y]_\k+\omega(x,D'y)JA+[x,D'y]_\k\\
    & =[D'x,y]_\k+[x,D'y]_\k,
\end{align*}
since $\omega(D'x,y)=-\omega(x,D'y)$.

\medskip

Therefore we have associated to any unimodular solvable Vaisman Lie algebra a K\"ahler flat Lie algebra equipped with a skew-symmetric derivation which commutes with the complex 
structure. This is summarized in the following theorem.

\smallskip

\begin{teo}\label{Main Theorem-A}
Let $\g$ be a unimodular solvable Lie algebra equipped with a Vaisman structure $(J,\pint)$, with fundamental $2$-form $\omega$ and Lee form $\theta$. Then there exists a K\"ahler 
flat Lie algebra $\k$ such that $\ker\theta=\k_{\omega'}(JA)$ and $\g=\k(D,\omega')$, where $D:=\ad_A|_{\ker\theta}$ and $\omega':=\omega|_{\k\times\k}$ is the fundamental 
$2$-form of $\k$. Moreover, $D':=D|_\k$ is a skew-symmetric derivation of $\k$ that commutes with its complex structure.
\end{teo} 

\

Next, we will prove the converse of Theorem \ref{Main Theorem-A}, namely, we show that beginning with a K\"ahler flat Lie algebra and a suitable derivation we are able to produce a 
Vaisman structure on a double extension of this Lie algebra.

\begin{teo}\label{Main Theorem-B}
 Let $(\k,J',\pint')$ be a K\"ahler flat Lie algebra with $\omega'$ its fundamental $2$-form and let $D'$ be a skew-symmetric derivation of $\k$ such that $J'D'=D'J'$. Let $D$ be 
the 
skew-symmetric derivation of the central extension $\k_{\omega'}(B)$ defined by: $D(B)=0$, $D|_{\k}=D'$. Then the double extension $\g:=\k(D,\omega')=\R A\ltimes_D 
\k_{\omega'}(B)$ 
admits a Vaisman structure $(J,\pint)$, where $JA=B$, $J|_\k=J'$, and $\pint$ extends $\pint'$ in the following way: $|A|=|B|=1$, $\la A,B\ra=0, \, \la A,\k\ra=\la B,\k\ra=0$. 
The Lie algebra $\g$ is unimodular and solvable, and the Lee form $\theta$ is the metric dual of $A$.
\end{teo}

\begin{proof}
Recall that the Lie bracket of $\g$ is given by: $\ad_A|_{\k}=D'$, $B\in\z(\g)$ and for any 
$x,y\in\k$, 
\[[x,y]=\omega'(x,y)B + [x,y]_\k,\]
where $\omega'(x,y)=\langle J'x,y\rangle'$ and $[\cdot,\cdot]_\k$ denotes the Lie bracket on $\k$. It is easy to see that $(J,\pint)$ in the statement is an almost Hermitian 
structure on $\g$, and we will call $\omega$ its K\"ahler $2$-form. Since $(\k,\pint')$ is flat then $\k$ is unimodular and solvable (see Remark \ref{k_unimodular} below), and it 
follows from Lemma \ref{unimod} that $\g$ is unimodular and solvable as well. 

\smallskip

Now we show that $J$ is a complex structure on $\g$. It is enough to show that $N_J(x,y)=0$ and $N_J(A,y)=0$ for all $x, y\in\k$. Firstly, 
\begin{align*} 
N_J(x,y) & = [Jx,Jy] -[x,y] -J([Jx,y] +[x,Jy])\\
    & = N_{J'}^{\mathfrak k}(x,y)+ \omega'(Jx,Jy)B-\omega'(x,y)B  - J(\omega'(x,Jy)B+\omega'(Jx,y)B)\\
    & = N_{J'}^{\mathfrak k}(x,y) + (-\la x,J'y\ra'-\la J'x,y\ra')B + (\la J'x,J'y\ra'-\la x,y\ra')A\\
    & = 0,
\end{align*}
for all $x, y\in\mathfrak k$ since $J'$ is a complex structure on $\k$, i.e. $N_{J'}^{\mathfrak k}=0$. Secondly,
\begin{align*} 
N_J(A,y) & = [JA,Jy] -[A,y] -J([JA,y] + [A,Jy])\\
    & = -Dy-JDJy\\
    & = -D'y-J'D'J'y\\
    & = 0,
\end{align*}
for any $y\in\k$ since $JA=B$ is central and $D'$ commutes with $J'$. Thus $J$ is integrable on $\g$.

\smallskip

The next step is to show that $(J,\pint)$ is LCK, to do this we have to verify that $d\omega=\theta\wedge\omega$ where $\omega$ is the fundamental form on $\g$ and $\theta$ is the 
dual $1$-form of the vector $A$.

\smallskip

If $x,y,z\in\k$, then 
\begin{align*} 
d\omega(x,y,z) & = -\omega([x,y],z)-\omega([y,z],x)-\omega([z,x],y)\\
    & = \la [x,y]_\mathfrak k +\omega'(x,y)B,Jz\ra+ \la [y,z]_\mathfrak k +\omega'(y,z)B,Jx\ra+ \la [z,x]_\mathfrak k +\omega'(z,x)B,Jy\ra\\
    & = \la [x,y]_\mathfrak k ,J'z\ra'+ \la [y,z]_\mathfrak k ,J'x\ra'+ \la [z,x]_\mathfrak k +,J'y\ra'\\
    & = d^\k\omega'(x,y,z)\\
    & = 0,
\end{align*}
where $d^\k$ denotes the differential on $\k$ and we have used that $d^\k\omega'=0$. On the other hand, $\theta\wedge\omega(x,y,z)=0$ since $\mathfrak k\subset \ker\theta$.

\smallskip

If $y,z\in\k$, then 
\begin{align*} 
d\omega(B,y,z) & = -\omega([B,y],z)-\omega([y,z],B)-\omega([z,B],y)\\
    & = \la [y,z]_\mathfrak k +\omega'(y,z)B,JB\ra \\
    & = 0.
\end{align*}
On the other hand, $\theta\wedge\omega(B,y,z)=0$ since $\R B\oplus\mathfrak k = \ker\theta.$

If $y,z\in\k$, then
\begin{align*} 
d\omega(A,y,z) & = -\omega([A,y],z)-\omega([y,z],A)-\omega([z,A],y)\\
    & = \la Dy,Jz\ra+ \la \omega'(y,z)B+[y,z]_\k ,JA\ra+ \la -Dz,Jy\ra\\
    & = \la D'y,J'z\ra' + \omega'(y,z) -\la D'z,J'y\ra'\\
%    & = \la -JDY,Z\ra + \omega(Y,Z) + \la Z,DJY\ra \\
    & = \omega'(y,z),
\end{align*} 
since $D'$ is skew-symmetric and commutes with $J'$. On the other hand, $\theta\wedge\omega(A,y,z)=\omega(y,z)=\omega'(y,z)$.

If $z\in\k$, then 
\begin{align*} 
d\omega(A,B,z) & = -\omega([A,B],z)-\omega([B,z],A)-\omega([z,A],B)\\
    & = \la D'z,A\ra' \\
    & = 0,
\end{align*} 
since $B$ is central and $D'z\in\k$.
On the other hand, $\theta\wedge\omega(A,B,z)=\omega(B,z)=0$.

Thus we have that $(J,\pint)$ is an LCK structure on $\g$. Moreover, $(J,\pint)$ is Vaisman since $\ad_A=D$ is a skew-symmetric endomorphism of $\g$ (Proposition 
\ref{adA-antisim}).
\end{proof}

\medskip

As a by-product of this analysis, we obtain the following stronger version of Corollary \ref{sasaki-centro}:

\begin{cor}\label{sasaki-centro2}
 Any unimodular solvable Lie algebra admitting a Sasakian structure is a central extension of a K\"ahler flat Lie algebra.
\end{cor}

\

It follows from Theorems \ref{Main Theorem-A} and \ref{Main Theorem-B} that there is a one-to-one correspondence between unimodular solvable Lie algebras equipped with a Vaisman 
structure and pairs $(\k, D')$ where $\k$ is a K\"ahler flat Lie algebra and $D'$ is a skew-symmetric derivation of $\k$ which commutes with the complex structure. Moreover, 
Theorem \ref{Main Theorem-B} provides a way to construct all unimodular solvable Lie algebras carrying Vaisman metrics. In order to do this, we need a better understanding of the 
derivations of K\"ahler flat Lie algebras, and we pursue this in the following section.

\

\section{Derivations of K\"ahler flat Lie algebras}\label{kahler-flat}

Let us recall the following result which describes the structure of any Lie algebra equipped with a flat metric. The original version was proved by Milnor in \cite{Mi}, and it 
was later refined in \cite{BDF} (compare with \cite{Al}). 

\begin{prop}[\cite{Mi}, \cite{BDF}]\label{gplana}
	Let $(\k,\pint)$ be a flat Lie algebra. Then $\k$ decomposes orthogonally as $\k=\z\oplus\h\oplus\k'$ (direct sum of vector spaces) where $\z$ is the center of 
$\k$ and the following properties are satisfied: 
	\begin{enumerate}[(a)]
		\item $\k'=[\k,\k]$ and $\h$ are abelian.
		\item $\ad :\h\to\mathfrak{so(k')}$ is injective and $\k'$ is even dimensional. In particular, $\dim \h\leq \frac{\dim \k'}{2}$.
		\item $\ad_x=\nabla_x$ for any $x\in\z\oplus\h$.
		\item $\nabla_x=0$ if and only if $x\in\z\oplus\k'$.
	\end{enumerate}
\end{prop}

\begin{rem}\label{k_unimodular}
	It follows easily from Proposition \ref{gplana} that $\k$ is a unimodular solvable Lie algebra, whose nilradical is given by $\z\oplus \k'$.
\end{rem}

%According to Proposition \ref{gplana} the  flat Lie algebra 
%$(\k,[\cdot,\cdot]_\k)$ of the Theorem \ref{Main Theorem} can be written as 
%\[\k=\mathfrak{z}\oplus\mathfrak{h}\oplus\k',\]
%where $\mathfrak{z}$ is the center of $\k$ and $\mathfrak{k'}$ is the commutator ideal. Moreover, $\z\oplus\mathfrak{k'}$ is the nilradical of $(\k, [\cdot,\cdot]_\k)$.

\medskip

We will use this proposition in order to better describe K\"ahler flat Lie algebras. First we give necessary and sufficient conditions for an almost complex structure on a flat 
Lie 
algebra to be K\"ahler (cf. \cite{Lic}).

\begin{prop}\label{Jkahler-sii}
Let $(\k,\pint)$ be a flat Lie algebra and let $J$ be an almost complex structure on 
$\k$, compatible with $\pint$. Then $J$ is K\"ahler if and only if the following two 
properties are satisfied:
\begin{enumerate}[(i)]
 \item $\mathfrak{z}\oplus\mathfrak{h}$ and $\k'$ are $J$-invariant.
 \item $\ad_H\circ J=J\circ\ad_H$, for any $H\in\h$.
\end{enumerate}
\end{prop}

\begin{proof}
Assume first that $J$ is K\"ahler, that is, $\nabla J=0$, or equivalently $\nabla_x J=J \nabla_x$ for any $x\in \k$.

For $x\in\mathfrak{k'}$, it follows from Proposition \ref{gplana} that $x=\displaystyle\sum_i [H_i,y_i]$ for some $H_i\in\mathfrak{h}$ and $y_i\in\mathfrak{k'}$. Then
\[Jx=\displaystyle\sum_i J[H_i,y_i]=\displaystyle\sum_i J\nabla_{H_i}y_i=\displaystyle\sum_i \nabla_{H_i}Jy_i= \displaystyle\sum_i [H_i,Jy_i]\in\mathfrak{k'},\]
where we have used Proposition \ref{gplana}(c). Therefore $\k'$ is $J$-invariant. Thus $J$ also preserves the orthogonal complement of $\mathfrak{k'}$, that is, 
$\mathfrak{z}\oplus\mathfrak{h}$ is $J$-invariant, and this proves $\ri$. Finally $\rii$ follows from the fact that $\nabla_H J=J \nabla_H$ and $\nabla_H=\ad_H$ for any 
$H\in\mathfrak{h}$.

\smallskip

If we assume now that $\ri$ and $\rii$ hold then it follows easily from Proposition \ref{gplana} that $\nabla_x J=J \nabla_x$ for any $x\in \k$, i.e. $J$ is K\"ahler.
\end{proof}

\medskip

\begin{rem}
Given any even-dimensional flat Lie algebra, it is easy to define an almost complex structure which satisfies the conditions of Proposition \ref{Jkahler-sii}, thus it becomes a 
K\"ahler flat Lie algebra (see \cite{Lic,BDF}).
\end{rem}

\

Our next aim is to study unitary derivations of a K\"ahler flat Lie algebra, i.e. skew-symmetric derivations which commute with the complex structure. In order to do so, we prove the 
following lemma.

\begin{lema}\label{Dh0}
Let $(\k,\pint)$ be a flat Lie algebra with decomposition $\k=\z\oplus\h\oplus\k'$ as in Proposition \ref{gplana}. If $D$ is a skew-symmetric derivation of $\k$, then $D(\h)=0$.
\end{lema}

\begin{proof}
Since $D$ is a derivation of $\k$, then $D(\k')\subset \k'$ and $D(\z)\subset\z$. Moreover, since $D$ is skew-symmetric and the sum is orthogonal we have that $D(\h)\subset\h$. 
	
It follows from Proposition \ref{gplana}(b) that $\ad_H\in\mathfrak{so(\k')}$ for any $H\in\h$. Then, since $\h$ is abelian, we get that  
$\mathfrak{F}=\{\ad_H:\k'\to\k': H\in\h\}$ is a commutative family of skew-symmetric endomorphisms of $\k'$. Therefore, $\mathfrak{F}$ is contained in a maximal abelian subalgebra 
$\mathfrak a$ of $\mathfrak{so}(\mathfrak k')$. The subalgebra $\mathfrak a$ is conjugated by an element of $SO(\k')$ to the 
following maximal abelian subalgebra of $\mathfrak{so}(\k')$: 
	\[\left\{\left(\begin{array}{ccccc}       
	0 &-a_1    &        &    &    \\
	a_1 & 0    &        &    &    \\
	&    & \ddots &    &    \\
	&    &        &  0 & -a_n \\
	&    &        &  a_n &  0 \\
	\end{array}\right): a_i\in\R \right\},\]
for some orthonormal basis $\{e_1,f_1,\dots,e_n,f_n\}$ of $\k'$. In this basis the elements of the family $\mathfrak{F}$ can be represented by matrices 
	\[
	\ad_H=\left(\begin{array}{ccccc}       
	0  & -\lambda_1(H) &        &               &              \\
	\lambda_1(H) & 0          &        &               &              \\
	&               & \ddots &               &              \\
	&               &        &         0     & -\lambda_n(H)\\
	&               &        &  \lambda_n(H) &     0        \\
	\end{array}\right), \]
for some $\lambda_i\in\h^*$,  $i=1,\dots,n$. 	We compute \[D[H,e_i]=[DH,e_i]+[H,De_i]=\lambda_i(DH)f_i+[H,De_i],\] 
while, on the other hand, \[D[H,e_i]=D(\lambda_i(H)f_i)=\lambda_i(H)Df_i.\]
Comparing the components in the direction of $f_i$ in both expressions we obtain that \[0=\langle\lambda_i(H)Df_i,f_i\rangle= \langle\lambda_i(DH)f_i+[H,De_i],f_i\rangle= 
\lambda_i(DH),\] 
since $D$ and $\ad_H$ are skew-symmetric. Therefore, $\lambda_i(DH)=0$ for all $i$, that is, $\ad_{DH}=0$. It follows from Proposition \ref{gplana}(b) that $DH=0$.
\end{proof}

\

As a consequence we have the following result, which will be an important tool to produce examples of unimodular solvable Lie algebras with Vaisman structures in \S\ref{examples}.

\medskip

\begin{teo}\label{base linda}
If $(\mathfrak k, J,\pint)$ is a K\"ahler flat Lie algebra and $D$ is a unitary derivation of $\k$, then $D(\h+J\h)=0$ and there exists an orthonormal basis 
$\{e_1,f_1,\dots,e_n,f_n\}$ of $\k'$ such that 
\[J|_{\k'}=\left(\begin{array}{ccccc}
    0 &-1    &        &    &    \\
    1 & 0    &        &    &    \\
        &    & \ddots &    &    \\
        &    &        &  0 & -1 \\
        &    &        &  1 &  0 \\
\end{array}\right), \quad
D|_{\k'}=\left(\begin{array}{ccccc}       

    0 &-a_1    &        &    &    \\
    a_1 & 0    &        &    &    \\
        &    & \ddots &    &    \\
        &    &        &  0 & -a_n \\
        &    &        &  a_n &  0 \\
\end{array}\right),\]
\[
\ad_H|_{\k'}=\left(\begin{array}{ccccc}       

       0  & -\lambda_1(H) &        &               &              \\
 \lambda_1(H) & 0          &        &               &              \\
          &               & \ddots &               &              \\
          &               &        &         0     & -\lambda_n(H)\\
          &               &        &  \lambda_n(H) &     0        \\
\end{array}\right), \]
for some $a_i\in\R$ and $\lambda_i\in\h^*$ for all $i=1,\dots,n$. These linear functionals satisfy $\lambda_i\neq 0$ for all $i=1,\ldots,n$, and $\bigcap_{i=1}^n\ker 
\lambda_i=\{0\}$.
\end{teo}

\begin{proof} 
It follows immediately from Lemma \ref{Dh0} and $DJ=JD$ that $D(\h+J\h)=0$.  

From Lemma \ref{Dh0} and the fact that $D$ is a derivation of $\k$ we obtain that $D$ commutes with $\ad_H$ for all $H\in\h$. Also, Proposition \ref{Jkahler-sii} implies 
that $J|_{\k'}$ commutes with $\ad_H|_{\k'}$ for all $H\in\h$. Therefore the family  
$\mathfrak{F}'=\{\ad_H:\k'\to\k': H\in\h\}\cup \{J|_{\k'},D|_{\k'}\}$ is a commutative family of skew-symmetric endomorphisms of $\k'$. The existence of the basis in the statement 
follows as in the proof of Lemma \ref{Dh0}.

We analyze next the linear functionals $\lambda_i\in\h^*$, $i=1,\ldots,n$. If $\lambda_i=0$ for some $i$ then $[H,e_i]=0=[H,f_i]$ for all $H\in\h$, and this implies that 
$e_i,f_i\in\z$, which is a contradiction since $\z\cap\k'=\{0\}$. Now, if $H\in \bigcap_{i=1}^n\ker \lambda_i=\{0\}$, we have that $\ad_H|_{\k'}=0$ and therefore $H=0$ according to 
Proposition \ref{gplana}(b).
\end{proof}	

\

\section{Further properties of Vaisman Lie algebras}
In this section we continue our study of unimodular solvable Vaisman Lie algebras, applying the results obtained in \S \ref{kahler-flat} to the K\"ahler flat Lie algebra given in 
Theorem \ref{Main Theorem-A}. In this way we obtain algebraic restrictions for the existence of Vaisman structures.

\medskip

In Corollary \ref{JA-central} we determined the center of a unimodular solvable Lie algebra $\g$ admitting a Vaisman structure. In what follows we will derive other 
algebraic properties of these Lie algebras, in particular we will analyze its commutator $\g'$ and its nilradical $\mathfrak{n}$. Recall that, since $\g$ is solvable, its 
nilradical is given by $\mathfrak n=\{x\in\g \, : \, \ad_x:\g\to\g \text{ is nilpotent}\}$ and $\g'\subseteq \n$.

Let us set some notation. We denote by $\u$ the largest $J$-invariant subspace of the center of $\k$, that is, $\u=\z\cap J\z$, and we define
$2r=\dim \u +\dim\mathfrak{k'}$ and $s=\dim \z - \dim \u$. 
Note that $\z\oplus\h$ decomposes orthogonally as $\z\oplus\h=\u\oplus(\h+J\h)$.

\smallskip

\begin{prop}\label{g'-n}
	Let $\g$ be a unimodular solvable Lie algebra admitting a Vaisman structure and consider the decomposition $\g=\R A\ltimes_D(\R JA\oplus_{\omega'}\k)$ with 
	$\k=\z\oplus\h\oplus\mathfrak k'$ as in Proposition \ref{gplana}. Then the commutator ideal $\g'$ of $\g$ is given by $\g'=\R JA\oplus 
\operatorname{Im}(D|_\u)\oplus\mathfrak{k'}$, while the nilradical $\n$ of $\g$ is given by:
	\begin{itemize}
	 \item If $D|_\u\neq0$ or $D|_{\k'}\notin \ad(\h)\subset \mathfrak{so}(\k') $, then $\n=\R JA\oplus\z\oplus\mathfrak{k'}\simeq \R^{s}\times\h_{2r+1}$. 
	 
	 \item If $D|_\u=0$ and $D|_{\k'}=0$, i.e., $A\in\z(\g)$, then $\n=\R A\oplus\R JA\oplus\z\oplus\mathfrak{k'}\simeq \R^{s+1}\times\mathfrak{h}_{2r+1}.$
	 
	 \item If $D|_\u=0$ and $0\neq D|_{\mathfrak k'}\in \ad(\h)$, i.e., $D|_{\mathfrak k'}=-\ad_H|_{\k'}$ for a unique $0\neq H\in \h$, then 
$\n=\R(A+H)\oplus\R JA\oplus\z\oplus\k'$. Furthermore, if $JH\in\h$, then $\n\simeq \R^{s+1}\times\h_{2r+1}$; and if $JH\notin\h$, then $\n\simeq \R^{s-1}\times\h_{2(r+1)+1}$.
	\end{itemize}
\end{prop}

\begin{proof}
Since $JA\in \g'$ (see Proposition \ref{JA-conm}), it follows from \eqref{corchetek} that $\k'=[\k,\k]_\k\subset \g'$. Clearly, 
$D(\h+J\h)=(D|_{\k})(\h+J\h)$ and, as $D|_{\k}$ is a unitary derivation of $\k$ (Theorem \ref{Main Theorem-A}), we have 
$D(\h+J\h)=0$, due to Theorem \ref{base linda}. As a consequence,  $\g'=\R JA\oplus\operatorname{Im}(D|_\u)\oplus\mathfrak{k'}$. 
	
	In order to determine the nilradical note first that, for any $z\in\z$, $\operatorname{Im}(\ad_z)\subset\R JA\oplus\mathfrak{z}$ and 
$\operatorname{Im}(\ad_z^2)\subset\R JA$. Since $JA\in\z(\g)$, we obtain that $\ad_z$ is a nilpotent operator on $\g$, which implies that $\z\subset \mathfrak{n}$. Therefore $\R 
JA\oplus\z\oplus\mathfrak{k'}\subset\mathfrak{n}$. Moreover, it follows from Proposition \ref{gplana}(b) that $\h\cap\n= \{0\}$.
	
Suppose that $\R JA\oplus\z\oplus\mathfrak{k'}\subsetneq\mathfrak{n}$, then there exists $0\neq A+H\in\n$ for some $H\in\h$. Therefore the operator $\ad_{A+H}$ is nilpotent. This 
operator can be written, for some orthonormal basis of $\R JA\oplus\u\oplus(\h+J\h)\oplus\mathfrak k'$, as 
\[\ad_{A+H}=\left(\begin{array}{c|ccc|ccc|ccc}       
\; &  &     & &    & v^t &   &  &  & \\
\hline
&  &     & &    &    &    &  &  &   \\
&  & D|_\u  & &    &    &    &  &  &   \\
&  &     & &    &    &    &  &  &   \\
\hline
&  &     & &    &    &    &  &  &   \\
&  &     & &    &    &    &  &  &   \\
&  &     & &    &    &    &  &  &   \\
\hline
&  &     & &    &    &    &  &  &   \\
&  &     & &    &    &    &  & D|_{\mathfrak k'} +\ad_H|_{\mathfrak k'} &   \\
&  &     & &    &    &    &  &  &   \\
\end{array}\right),\]
with both $D|_\u$ and $D|_{\mathfrak k'} +\ad_H|_{\mathfrak k'}$ skew-symmetric. It follows that $\ad_{A+H}$ is nilpotent if and only if
\begin{equation}\label{nilrad}
D|_\u=0 \quad \text{ and } \quad D|_{\mathfrak k'}=- \ad_H|_{\mathfrak k'}.
\end{equation}
It follows that $H\in \h$ satisfying \eqref{nilrad} is unique, since $\ad:\h\to \mathfrak{so}(\k')$ is injective (Proposition \ref{gplana}(b)). We have two possibilities: if $H=0$, 
then $A\in\z(\g)$ and $\n=\R A\oplus\R JA\oplus\z\oplus\mathfrak{k'}$. On the other hand, if $H\neq0$, then $\n=\R(A+H)\oplus\R JA\oplus\z\oplus\k'$.

\smallskip

Assume now that $\n=\R JA\oplus\z\oplus\mathfrak{k'}$; according to \eqref{nilrad} this happens if and only if $D|_\u \neq0$ or $D|_{\k'}\notin\ad(\h)$.

\smallskip

The isomorphism classes of $\n$ in the different cases above follow easily from the description of the Lie bracket of $\g$.
\end{proof}

\medskip

\begin{rem} 
Note that the nilradical $\n$ in Proposition \ref{g'-n} is isomorphic to $\R^{p-1}\times\mathfrak{h}_{2q+1},$ for some $p,q\in\N$ and, as a consequence, it is never abelian. On 
the other hand, it follows from Proposition \ref{g'-n} and Lemma \ref{Dh0} that $\n^\perp$ is an abelian subalgebra of $\g$ if and only if $\h\cap J\h=\{0\}$.
\end{rem}

\medskip

We can give next a strong algebraic obstruction to the existence of Vaisman structures on a unimodular solvable Lie algebra.

\begin{teo}\label{Vaisman-imag.puras}
If the unimodular solvable Lie algebra $\g$ admits a Vaisman structure, then the eigenvalues of the operators $\ad_x$ with $x\in\g$ are all imaginary (some of them are 0).
\end{teo}

\begin{proof}
Recall that $\g$ is a double extension $\g=\R A\ltimes_D(\R JA\oplus_{\omega'}\k)$ with $\k=\z\oplus\h\oplus\mathfrak k'$ and $\z\oplus\h=\z\cap 
J\z\oplus(\h+J\h)$. Let $\{A,JA,u_1,v_1,\dots,u_p,v_p,x_1,y_1,\dots,x_q,y_q,e_1,f_1,\dots,e_m,f_m,\}$ be an orthonormal basis of $\g$ adapted to this decomposition. Moreover, 
$Ju_i=v_i$, $Jx_i=y_i$ and $Je_i=f_i$.

Given $x=aA+bJA+z+h+y\in\g$ with $a,b\in\R$, $z\in\z\cap J\z$, $h\in\h+J\h$,  $y\in\k'$,  we compute the matrix of the operator $\ad_x$ in such basis:
\[\ad_x=\left(\begin{array}{c|c|ccc|ccc|ccc}       
   &  &  &     & &  &    &  &  &    &   \\
\hline 
   &  &  & \alpha^t    & &  & \beta^t   &  &  &    & \\
\hline
   &  &  &     & &  &    &  &  &    &   \\
  \gamma  &  &   & aD|_{\z\cap J\z}  &  &    &  &  &    &   \\
   &  &  &     & &  &    &  &  &    &   \\
\hline
   &  &  &     & &  &    &  &  &    &   \\
   &  &  &     & &  &    &  &  &    &   \\
   &  &  &     & &  &    &  &  &    &   \\
\hline
   &  &  &     & &  &    &  &  &    &   \\
 \delta  &  &     & &  &    & C&  & & aD|_{\mathfrak k'} + B &   \\
   &  &  &     & &  &    &  &  &    &   \\
\end{array}\right),\]
for some $\alpha\in\R^{2p}$, $\beta,\gamma\in\R^{2q}$, $\delta\in\R^{2m}$, $C\in \operatorname{Mat}(m\times q,\R)$ and $B\in \mathfrak{u}(m)$.
Since $aD|_{\z\cap J\z}$ and $aD|_{\mathfrak k'} + B$ are skew-symmetric, it is easy to see that the eigenvalues of $\ad_x$ for $x\in\g$ are all imaginary.
\end{proof}

\medskip

\begin{rem}
A Lie algebra $\g$ satisfying the condition that the eigenvalues of the operators $\ad_x$ are all imaginary for all $x\in\g$ is called a Lie algebra of type I (see 
\cite{O-V}).
\end{rem}

\

A consequence of Theorem \ref{Vaisman-imag.puras} is the following result, proved recently by H. Sawai in \cite{S3}. We recall that a Lie group is called 
completely solvable if it is solvable and all the eigenvalues of the adjoint operators $\ad_x$ are real, for all $x$ in its Lie algebra.

\begin{cor}
Let $G$ be a simply connected completely solvable Lie group and $\Gamma\subset G$ a lattice. If the solvmanifold $\Gamma\backslash G$ admits a Vaisman structure $(J,g)$ such that 
the complex structure $J$ is induced by a left invariant complex structure on $G$, then $G=H_{2n+1}\times\R$, where $H_{2n+1}$ denotes the $(2n+1)$-dimensional Heisenberg Lie 
group.
\end{cor}

\begin{proof}
Using the symmetrization process of Belgun (\cite{Bel}) together with results in \cite{S2}, one can produce a left invariant Riemannian metric $\tilde{g}$ on $G$ such that 
$(J,\tilde{g})$ is again a Vaisman structure. This gives rise to a Vaisman structure on $\operatorname{Lie}(G)$ and, taking into account Theorem \ref{Vaisman-imag.puras} and the fact that $G$ is completely solvable, we have 
that $G$ is a nilpotent Lie group. It follows from \cite{S} that $G=H_{2n+1}\times\R$.
\end{proof}

\medskip

In the next section we will show the existence of Vaisman structures on solvmanifolds $\Gamma\backslash G$ where $G$ is not completely solvable, in any even dimension. 

\

\section{Examples}\label{examples}

Using the results in the previous sections we will construct many examples of unimodular solvable non-nilpotent Lie algebras equipped with Vaisman structures. In particular, we 
will provide two infinite families of such examples, and also the classification of such Lie algebras in dimensions $4$ and $6$. We will also show the existence of families of
lattices in these examples, obtaining in this way compact Vaisman solvmanifolds which are not diffeomorphic to the well known nilmanifolds arising from the Heisenberg groups 
(Example \ref{heisenberg}). 

\smallskip

Firstly, we would like to prove a general result concerning lattices in semidirect products, which will be frequently used later in this section.

\begin{lema}\label{lattices-semidirecto}
Let $H$ be a simply connected Lie group equipped with a Lie group homomorphism $\phi: \R\to \operatorname{Aut}(H)$, and let $G=\R\ltimes_\phi H$ be the corresponding semidirect 
product. Let $\Gamma$ be a lattice in $H$. If 
there exists $a\in\R$, $a\neq0$, such that $\phi(a)(\Gamma)\subset\Gamma$, then $a\Z \ltimes_\phi \Gamma$ is a lattice in $G$.
\end{lema}

\begin{proof}
If $a\in\R$, $a\neq0$, satisfies $\phi(a)(\Gamma)\subset\Gamma$, then $\phi(ak)(\Gamma)\subset\Gamma$ for all $k\in\Z$ and therefore the semidirect product 
$\tilde{\Gamma}:=a\Z\ltimes_\phi \Gamma$ is well defined. Clearly, $\tilde{\Gamma}$ is a discrete subgroup of $G$. We only have to show that it is co-compact. 

Let us recall that two elements $(t_1,g_1)$ and $(t_2,g_2)$ in $G$ belong to the same right-coset with respect to $\tilde{\Gamma}$  if and only if there exist $k\in\Z$ 
and 
$\gamma\in\Gamma$ such that 
\begin{align}\label{t-gamma}
(t_2,g_2) =& (ak,\gamma)(t_1,g_1) \\
=& (ak + t_1, \gamma \phi(a)^kg_1). \nonumber
\end{align}

We may assume $a>0$. Let $\{a_n\}$ be a sequence in $\tilde{\Gamma}\backslash G$, with $a_n=[(t_n,g_n)]$. It follows easily from \eqref{t-gamma} that we can choose $t_n\in[0,a]$ for all $n\in\N$; 
moreover, as $\Gamma\backslash H$ is compact, we may assume that $[g_n]$ converges to $[g]$ in $\Gamma\backslash H$, for some $g\in H$.

The canonical projection $\pi: H\to \Gamma\backslash H$ is a local diffeomorphism, therefore we can choose a representative $g'_n$ of $[g_n]$ such that $g'_n$ converges to $g$ in 
$H$. Taking into account that also $[0,a]$ is compact, we may assume that $a_n=[(t_n,g'_n)]$ satisfies $t_n\to t$ in $[0,a]$ and $g'_n\to g$ in $H$. It is easy to see that 
$[(t_n,g'_n)]$ converges to $[(t,g)]$ in $\tilde{\Gamma}\backslash G$. Since $\{a_n\}$ is arbitrary, we thus obtain that  $\tilde{\Gamma}\backslash 
G$ is compact.
\end{proof}

\

\subsection{Example 1:}\label{ex-1}

We start with the abelian Lie algebra $\k=\mathbb{R}^{2n}$ with its canonical K\"ahler flat structure $(J,\pint)$. 
Let $\{e_1,f_1,\dots,e_n,f_n\}$ be an orthonormal basis of $\k$ where $Je_i=f_i$ and let $\omega=\sum_{i=1}^n e^i\wedge f^i$ be the fundamental form. 
Then we consider the central extension 
\[ \k_\omega(B)=\R B\oplus_\omega\k,\]
which is easily seen to be isomorphic to $\h_{2n+1}$, the $(2n+1)$-dimensional Heisenberg Lie algebra.

Next,  we consider the double extension
\[\g=\k(D,\omega)=\R A\ltimes_D\mathfrak{h}_{2n+1},\]
where the action is given by 
\begin{equation}\label{D} 
D=\left(\begin{array}{cccccc}       
0 &      &      &        &      &     \\
  &   0  & -a_1 &        &      &     \\
  &  a_1 & 0    &        &      &     \\
  &      &      & \ddots &      &     \\
  &      &      &        &   0  & -a_n\\
  &      &      &        &  a_n &  0  \\
\end{array}\right), \end{equation}
in the basis $\{B,e_1,f_1,\dots,e_n,f_n\}$ where $[e_i,f_i]=B$, for some $a_i\in\R$. Since $D|_{\k}$ satisfies the conditions of Theorem \ref{Main Theorem-B}, $\g$ admits a 
Vaisman structure. 

\smallskip

Notice that there is no loss of generality when we assume that $D$ is given as in \eqref{D}. Indeed, it can be seen that given any unitary derivation $D'$ of $(\R^{2n},J,\pint)$ there exists an orthonormal basis $\{e_1,f_1,\dots,e_n,f_n\}$ of $\R^{2n}$ such that $Je_i=f_i$ and $D'=D|_\k$ with $D$ is as in \eqref{D}.
 
\smallskip

If $a_i=0$ for all $i=1,\ldots,n$, then we recover the well known Vaisman structure on $\R\times\h_{2n+1}$ (see Example \ref{heisenberg}).

Let us consider from now on the case when not all the $a_i$'s vanish. We may reorder the elements of the basis of $\g$, if necessary, and we may assume that the 
constants $a_i$ satisfy $a_1\leq a_2\leq\cdots \leq a_n$. We will denote this Lie algebra by $\g_{(a_1,\ldots,a_n)}$. Note that 
$\g_{(a_1,\ldots,a_n)}$ is an almost nilpotent Lie algebra, i.e, it contains a codimension one nilpotent ideal. In particular, the nilradical of $\g_{(a_1,\ldots,a_n)}$ is the nilpotent ideal $\h_{2n+1}$. The associated simply connected Lie group will be denoted $G_{(a_1,\ldots,a_n)}$ and, when $0<a_1\leq 
a_2\leq\cdots \leq a_n$, it is known as an \textit{oscillator group}. Oscillator groups have many geometric properties, for instance, it was proved in \cite{M} that they are the 
only simply connected, non simple Lie groups which admit an indecomposable bi-invariant Lorentz metric. Other geometric features of oscillator groups in dimension $4$ have been 
studied in \cite{COS}. In this example we have shown: 

\smallskip

\begin{teo}
 Any solvmanifold of a group $G_{(a_1,\ldots,a_n)}$ admits invariant Vaisman structures. In particular any oscillator solvmanifold admits this kind of structures.
\end{teo}

\medskip

Lattices in oscillator groups have been characterized in \cite{Fi}. Here we will provide an explicit construction of some families of lattices in $G_{(a_1,\ldots,a_n)}$, and later 
we will determine the first homology group and first Betti number of the associated solvmanifolds.

We begin by recalling the following result, which deals with the isomorphism classes of these Lie algebras.

\begin{lema}[\cite{MR}]\label{rescaling}
Let $\g_{(a_1,\ldots,a_n)}$ and $\g_{(b_1,\ldots,b_n)}$ be two Lie algebras as defined above. If there exists $c\in\R-\{0\}$ such that $a_j=cb_j$ for all $j=1,\ldots,n$ then 
$\g_{(a_1,\ldots,a_n)} \cong \g_{(b_1,\ldots,b_n)}$. 
\end{lema}

%\begin{rem}
% In \cite{Fi} it is proved that if $0 < a_1\leq a_2\leq\cdots \leq a_n$ and $0 < b_1\leq b_2\leq\cdots \leq b_n$ then $\g_{(a_1,\ldots,a_n)} \cong \g_{(b_1,\ldots,b_n)}$ if and only if there exists $c>0$ such that $a_j=cb_j$ for all $j=1,\ldots,n$.
%\end{rem}

\

The Lie group $G_{(a_1,\ldots,a_n)}$ can be described as follows. We consider the Lie group homomorphism $\varphi: \R \longrightarrow 
\operatorname{Aut}(H_{2n+1})$ given by
\[ \varphi(t)=e^{tD}=\left(\begin{array}{cccccc}       
1 &             &           &        &           &     \\
  &   \cos(a_1t) &-\sin(a_1t) &        &           &     \\
  &   \sin(a_1t) & \cos(a_1t) &        &           &     \\
  &             &           & \ddots &           &     \\
  &             &           &        & \cos(a_nt) &-\sin(a_nt)\\
  &             &           &        & \sin(a_nt) &  \cos(a_nt)\\
\end{array}\right)  \]
where $H_{2n+1}$ is the $(2n+1)$-dimensional Heisenberg Lie group, i.e. the Euclidean manifold $\R^{2n+1}$ equipped with the following product:
\[(z,x_1,y_1,\dots,x_n,y_n)\cdot (z',x'_1,y'_1,\dots,x'_n,y'_n)=(z+z'+\frac 12 \sum_{j=1}^n(x_jy'_j-x'_jy_j),x_1+x'_1,\dots,y_n+y'_n).\]
Then $G_{(a_1,\ldots,a_n)}$ is the semidirect product $G_{(a_1,\ldots,a_n)}=\R\ltimes_\varphi H_{2n+1}$.

\medskip 

We show next the existence of lattices for some choice of the parameters $a_i$ (compare \cite{MR}):

\begin{prop}\label{lattices-oscillator}
If $a_i\in\Q$ for $i=1,\ldots,n$, then $G_{(a_1,\ldots,a_n)}$ admits lattices. 
\end{prop}

\begin{proof}
If $a_i\in\mathbb{Q}$, then $a_i=\frac{p_i}{q_i}$ for some $p_i\in\Z,\,q_i\in\N$ 
with $(p_i,q_i)=1$. Setting $t_0:=2\pi\prod q_i$, we obtain that $\varphi(t_0)$ is an integer matrix. Moreover, the structure constants of $\g_{(a_1,\ldots,a_n)}$ corresponding to 
the basis $\{B,e_1,f_1,\dots,e_n,f_n\}$ are all rational. As $G_{(a_1,\ldots,a_n)}$ is an almost nilpotent Lie group, it follows from \cite{B} (see also \cite{FOU}) that 
$G_{(a_1,\ldots,a_n)}$ 
admits lattices.
\end{proof}

\medskip

Therefore we will consider $a_i\in\Q$ for all $i=1,\ldots, n$; moreover, it follows from Lemma \ref{rescaling} that we may assume $a_i\in\Z$ for all $i=1,\ldots, n$, 
with $\operatorname{gcd}(a_1,\ldots,a_n)=1$.

\smallskip

Beginning with a lattice in $H_{2n+1}$ we may extend it to a lattice in $G_{(a_1,\ldots,a_n)}$. Consider the following lattices in $H_{2n+1}$: for each $k \in \N$ there exists a 
lattice $\Gamma_k$ in $H_{2n+1}$ given by $\Gamma_{k}=\frac{1}{2k}\Z\times\Z\times\cdots\times\Z$. It can be shown that $\Gamma_k/[\Gamma_k, \Gamma_k]$ is  isomorphic to $\mathbb 
Z^{2n}\oplus \mathbb Z_{2k}$. Hence, $\Gamma_r$ and $\Gamma_s$ are non-isomorphic for $r\neq s$. 

Any lattice $\Gamma_k$ in $H_{2n+1}$ is in\-va\-rian\-t under the subgroups generated by $\varphi(\pi/2)$, $\varphi(\pi)$ and $\varphi(2\pi)$. According to Lemma 
\ref{lattices-semidirecto} we have three families of lattices in $G_{(a_1,\ldots,a_n)}$:
\begin{align*}
 \Lambda_{k,\frac{\pi}{2}} & =  \frac{\pi}{2}\Z\ltimes_\varphi\Gamma_k,\\
 \Lambda_{k,\pi} & =  \pi\Z\ltimes_\varphi\Gamma_k,\\
 \Lambda_{k,2\pi} & = 2\pi\Z\ltimes_\varphi\Gamma_k.
\end{align*}

\medskip

We analyze next some topological properties of the solvmanifolds $\Lambda_{k,j}\backslash G_{(a_1,\ldots,a_n)}$ for $j=2\pi,\pi,\pi/2$:

\

$\bullet$ Note that $\varphi(2\pi)=\operatorname{Id}$, therefore $\Lambda_{k,2\pi}=2\pi\Z\times\Gamma_k$, which is isomorphic to a lattice in $\R\times H_{2n+1}$. According to 
Corollary \ref{mostow}, we have that the solvmanifold $\Lambda_{k,2\pi}\backslash G_{(a_1,\ldots,a_n)}$ is isomorphic to the nilmanifold $S^1\times \Gamma_k\backslash H_{2n+1}$, 
for any choice of $(a_1,\ldots,a_n)$. It is easy to see that the first homology group of this nilmanifold is $\Z^{2n+1}\oplus\Z_{2k}$ and hence its first Betti number is 
$b_1=2n+1$.

\

$\bullet$ For the family $\Lambda_{k,\pi}$, note that the isomorphism class of this lattice depends only on the parity of the integers $a_j$, therefore according to Corollary 
\ref{mostow}, 
we may assume that $a_j\in \{0,1\}$ for all $j$, not all of them equal to 0. After reordering, we have that there exists $p\in\{0,1,\ldots,n-1\}$ such that $a_j=0$ if $j\leq p$ 
and $a_j=1$ if $j>p$. Then it can be seen that 
\begin{align*} 
 [\Lambda_{k,\pi},\Lambda_{k,\pi}] & = \{ (0,r,s_1,t_1,\ldots,s_n,t_n)\in\Lambda_{k,\pi}\,: r\in\Z;\;   \\
                                   & \qquad \qquad  s_j=t_j=0\, (j=1,\ldots,p); \, s_j,t_j\in 2\Z\, (j=p+1,\ldots,n)\}, 
\end{align*}
thus the first homology group of the solvmanifold $\Lambda_{k,\pi}\backslash G_{(a_1,\ldots,a_n)}$ is 
\[ \Lambda_{k,\pi}/[\Lambda_{k,\pi},\Lambda_{k,\pi}] \cong \Z\oplus\Z_{2k}\oplus (\Z\oplus \Z)^p \oplus  (\Z_2\oplus \Z_2)^{n-p}.\]
The first Betti number is $b_1=2p+1$.

\

$\bullet$ For the family $\Lambda_{k,\pi/2}$, note that the isomorphism class of this lattice depends only on the congruence class of the integers $a_j$ modulo 4, therefore 
according to 
Corollary \ref{mostow}, we may assume that $a_j\in \{0,1,2,3\}$ for all $j$, not all of them even. After reordering, we have that there exist $c,d\in\{0,1,\ldots,n-1\}$ with 
$0\leq c+d\leq n-1$ such that $a_1=\cdots=a_c=0$, $a_{c+1}=\cdots=a_{c+d}=2$ and $a_j\in\{1,3\}$ for $j>c+d$. It can be seen that 
\begin{align*}
 [\Lambda_{k,\frac{\pi}{2}},\Lambda_{k,\frac{\pi}{2}}] & = \{(0,r,s_1,t_1,\ldots,s_n,t_n)\in\Lambda_{k,\frac{\pi}{2}}\,: r\in\Z,\;  s_j=t_j=0\, (j=1,\ldots,c);  \\
                                                      & \qquad \qquad \quad  s_j,t_j\in 2\Z\, (j=c+1,\ldots,c+d);\; s_j+t_j\in 2\Z \,(j>c+d) \}, 
\end{align*}
thus the first homology group of the solvmanifold $\Lambda_{k,\frac{\pi}{2}}\backslash G_{(a_1,\ldots,a_n)}$ is 
\[ \Lambda_{k,\frac{\pi}{2}}/[\Lambda_{k,\frac{\pi}{2}},\Lambda_{k,\frac{\pi}{2}}] \cong \Z\oplus\Z_{2k}\oplus (\Z\oplus \Z)^c \oplus (\Z_2\oplus \Z_2)^d \oplus (\Z_2)^{n-(c+d)}.\]
The first Betti number is $b_1=2c+1$.

\

Clearly, the first homology groups of a solvmanifold in one of the last two families is not isomorphic to the first homology group of the nilmanifolds in the first family. 
Therefore, the solvmanifolds in the last two families are not diffeomorphic to the nilmanifolds $S^1\times \Gamma_k\backslash H_{2n+1}$.

Note that for any $n\in \N$ and $r\in\{0,1,\ldots,n-1\}$, we can find a $(2n+2)$-dimensional Vaisman solvmanifold $\Lambda_{k,j}\backslash G_{(a_1,\ldots,a_n)}$ with first Betti 
number $b_1=2r+1$. 

It is clear that the solvmanifolds we have just constructed admit a Riemannian submersion to the circle $S^1$, with fibers isometric to the Sasakian nilmanifold 
$\Gamma_k\backslash H_{2n+1}$.

\medskip

\begin{rem}
In \cite{MP} the authors provide examples of compact Vaisman manifolds which are obtained as the total spaces of a principal $S^1$-bundle over coK\"ahler manifolds, and they show 
that they are diffeomorphic to solvmanifolds. The Lie algebras associated to these solvmanifolds are isomorphic to some $\g_{(a_1,\ldots,a_n)}$, but the lattices that they 
consider are different from ours.
\end{rem}

\bigskip

\subsection{Example 2}\label{ex-2}
We start with a K\"ahler flat Lie algebra $(\mathfrak k, J, \pint)$ such that $\dim\h=1$, where  $\k=\mathfrak{z}\oplus\mathfrak{h}\oplus\k'$ is the orthogonal 
decomposition of $\k$ given by Proposition \ref{gplana}. Let $\mathfrak{h}=\mathbb{R}H$ with $|H|=1$, and let us set $2m:=\dim \k'$ and $2l+1:=\dim\z$.
According to Proposition \ref{Jkahler-sii}, if $Z:=JH$ then $Z\in\z$ and there exists a $J$-invariant subspace $\mathfrak u$ of $\z$ such 
that  $\mathfrak{z}=\mathbb{R}Z\oplus^\perp\u$. If $\omega$ denotes the fundamental 2-form of $(J,\pint)$, then it is easy to verify that the Sasakian central extension of $\k$ by 
$\omega$ can be decomposed as: 
\[\k_\omega(B)=\R H\ltimes_M(\R Z\times\h_{2(m+l)+1}),\]
where $M$ is the operator given by the following matrix
\[ M=\left(\begin{array}{cccccccc}      
0 & 0 &  &    &      &        &      &     \\
1& 0 &  &    &      &        &      &     \\
  &   & 0_{2l\times 2l} &    &      &        &      &     \\
  &   &  &  0  & -a_1 &       &      &     \\
  &   &  &a_1 & 0    &        &      &     \\
  &   &  &    &      & \ddots &      &     \\
  &   &  &    &      &        &   0  & -a_m\\
  &   &  &    &      &        &  a_m &  0  \\
\end{array}\right),\]
in an orthonormal basis $\{Z,B,e_1,f_1,\dots,e_l,f_l,u_1,v_1,\dots,u_m,v_m\}$ such that:  
$\{e_1,f_1,\dots,e_l,f_l\}$ is a basis of $\u$, $\{u_1,v_1,\dots,u_m,v_m\}$ is a basis of $\mathfrak k'$
and $Je_i=f_i$, $Ju_i=v_i$. Moreover, $[e_j,f_j]=B$ and $[u_j,v_j]=B$ for all $j$. It follows from Proposition \ref{gplana} that $a_j\neq 0$ for all $j=1,\ldots,m$.

\medskip

Let $\g$ be the double extension 
\[\g=\k(D,\omega)=\R A \ltimes_D(\R H\ltimes_M(\R Z\times\h_{2(m+l)+1})),\] 
where $D$ is the derivation of $\k_\omega(B)$ given by
\[ D= \begin{pmatrix}
0 &   &   &       &           &   &           &     \\
  & 0 &   &        &           &   &           &     \\
  &   & 0 &          &           &   &           &     \\
  &   &   & 0       & -\alpha_1 &   &           &     \\
  &   &   &\alpha_1 &   0       &   &           &     \\
  &   &   &        &         &\ddots&        &     \\
  &   &   &        &           &   &        0  & -\alpha_{m+l}\\
  &   &   &        &           &   &  \alpha_{m+l} &  0  \\
     \end{pmatrix}, \]
for some $\alpha_j\in\R$, in the basis $\{H,Z, B,e_1,f_1,\dots,e_l,f_l,u_1,v_1,\dots,u_m,v_m\}$.
Since $D|_{\k}$ satisfies the conditions of Theorem \ref{Main Theorem-B} we have that $\g$ admits a Vaisman structure.

\medskip

Now we study lattices in the simply connected Lie group $G$ associated to $\g$ when $a_i, \alpha_i\in\Z$ for any $i$. Set $n:=m+l$ and consider the Lie 
group homomorphism $\psi: \R\longrightarrow \operatorname{Aut}(\R\times H_{2n+1})$ given by
\[ \psi(t)=e^{tM}=\begin{pmatrix}       
1 & 0 &               &            &             &         &            &            \\
t & 1 &               &            &             &         &            &            \\
  &   & \I_{2l\times 2l} &            &             &         &            &            \\
  &   &               & \cos(a_1t) & -\sin(a_1t) &         &            &            \\
  &   &               & \sin(a_1t) &  \cos(a_1t) &         &            &            \\
  &   &               &            &             &  \ddots &            &            \\
  &   &               &            &             &         & \cos(a_mt) & -\sin(a_mt)\\
  &   &               &            &             &         & \sin(a_mt) &  \cos(a_mt)\\
\end{pmatrix}.\]

On $\R^{2n+3}$ consider the algebraic structure given by the semidirect product of $\R$ and $\R\times H_{2n+1}$ via $\psi$ and we obtain the simply 
connected Lie group 
\[S=\R\ltimes_\psi (\R\times H_{2n+1}).\]
 
If $\Gamma_k$ is the lattice in $H_{2n+1}$ considered in the examples in \S \ref{ex-1}, then the group $L_k=a\Z\times\Gamma_k$ is a lattice in $\R\times H_{2n+1}$ for any 
$a\in\R$, $a\neq0$. In particular, according to Lemma \ref{lattices-semidirecto}, for each $j= 2\pi, \pi, \frac{\pi}{2}$ we obtain a lattice $\Gamma_{k,j}$ in $S$ defined by 
\[\Gamma_{k,j}=j\Z\ltimes(j^{-1}\Z\times\Gamma_k). \]

Now we consider the Lie group homomorphism $\varphi: \mathbb{R}\longrightarrow \operatorname{Aut}(S)$ given by
\[  \varphi(t)=e^{tD}=\left(\begin{array}{cccccccc}       
1  &   &  &           &                       &        &           &     \\
   & 1 &  &           &                       &        &           &     \\
   &   & 1&           &                       &        &           &     \\   
   &   &  & \cos(\alpha_1t) &-\sin(\alpha_1t) &        &           &     \\
   &   &  & \sin(\alpha_1t) & \cos(\alpha_1t) &        &           &     \\
   &   &  &          &                        & \ddots &           &     \\
   &   &  &         &                         &        & \cos(\alpha_nt) &-\sin(\alpha_nt)\\
   &   &  &         &                         &        & \sin(\alpha_nt) &  \cos(\alpha_nt)\\
\end{array}\right)  \]

On $\R^{2n+4}$ consider the algebraic structure given by the semidirect product of $\R$ and $S$ via $\varphi$ and we obtain the simply connected Lie group 
\[G=\R\ltimes_\varphi S.\]

Since $\Gamma_{k,j}$ is invariant under the subgroups generated by $\varphi(2\pi), \varphi(\pi), \varphi(\frac{\pi}{2})$, then for each $\Gamma_{k,j}$ we have three new lattices 
in 
$G$, given by
\[\Lambda_{k,j,i}=i\Z\ltimes\Gamma_{k,j},\]
for $i= 2\pi, \pi, \frac{\pi}{2}$. Therefore we get new examples of solvmanifolds 
\[M_{k,j,i}=\Lambda_{k,j,i}\backslash\R\ltimes_\varphi(\R\ltimes_\psi(\R\times H_{2n+1}))\]
equipped with a Vaisman structure arising from a Vaisman structure on the Lie algebra.

\medskip

The solvmanifolds we have just constructed admit a Riemannian submersion to the circle $S^1$, with fibers isometric to a Sasakian solvmanifold $\Gamma_{k,j}\backslash 
(\R\ltimes_\psi(\R\times H_{2n+1}))$.

\

\subsection{Classification of Vaisman Lie algebras in low dimensions}

In this section we determine all unimodular solvable Lie algebras of dimension $4$ and $6$ that admit a Vaisman structure, up to Lie algebra isomorphism.

\

(i) According to Theorem \ref{Main Theorem-A}, any $4$-dimensional unimodular solvable Lie algebra with a Vaisman structure is a double extension $\g=\k(D,\omega)$, where $\k=\R^2$.  It is 
easy to see that $\g$ has an orthonormal basis $\{A,B,e,f\}$ such that
\[ [A,e]=cf,\quad [A,f]=-cf, \quad [e,f]=B, \quad JA=B, \; Je=f, \qquad c\in\R.\]
These Lie algebras belong to the family of examples constructed in \S\ref{ex-1}. If $c=0$, then $\g$ is isomorphic to the nilpotent Lie algebra $\R\times\h_3$, equipped with its canonical Vaisman structure (see Example \ref{heisenberg}), and any nilmanifold obtained as a quotient of $\R\times 
H_3$ is a primary Kodaira surface. If $c\neq0$, then we may assume $c=1$ (via a Lie algebra isomorphism), and $\g$ is a solvable non-nilpotent Lie algebra $\R\ltimes\h_3$, and any solvmanifold obtained as a quotient of $\R\ltimes H_3$ is  either a primary or secondary Kodaira surface, depending on the lattice considered (see \cite{H} and \S\ref{ex-1}). 

\

(ii) According to Theorem \ref{Main Theorem-A}, any $6$-dimensional unimodular solvable Lie algebra with a Vaisman structure is a double extension $\g=\k(D,\omega)$, where $\k$ is 
a $4$-dimensional K\"ahler flat Lie algebra which decomposes orthogonally as $\k=\mathfrak{z}\oplus\mathfrak{h}\oplus\k'$, according to Proposition 
\ref{gplana}. In particular, the dimension of $\mathfrak{k'}$ must be $0$ or $2$ and, as a consequence, $\dim\mathfrak{h}$ is equal to $0$ or $1$. 
Therefore there are only two cases for $\k$, namely, $\k_1=\z=\R^4$ or $\k_2=\R Z\oplus\R H\oplus\R^2$ for some $Z\in\z,\, H\in\h$ with $JH=Z$. In the latter case, the action of $H$ on 
$\k'$ is given by a matrix 
$\begin{pmatrix}
 0 & -a \\ a & 0
\end{pmatrix}$, with $a\neq 0$, and it is easy to see that these Lie algebras are all isomorphic to the one with $a=1$. The Sasakian central extension of $\k_1$ 
is isomorphic to $\h_5$, while the Sasakian central extension of $\k_2$ is the Lie algebra $\mathfrak s_5$ generated by $\{B,Z,H,e,f\}$ with Lie brackets   
\[ [H,e]=f, \; [H,f]=-e,\; [H,Z]=[e,f]=B,\] 
where $\{e,f\}$ is an orthonormal basis of $\R^2$ with $Je=f$. It is a consequence of Corollary \ref{sasaki-centro2} that $\h_5$ and $\mathfrak{s}_5$ are the only unimodular solvable 
Lie algebras of dimension $5$ which admit Sasakian structures (this follows also from \cite{AFV}).

\medskip

The corresponding unimodular solvable Vaisman Lie algebras, obtained as double extensions of $\k_1$ and $\k_2$, are given in the next result. Note that the double extensions 
$\k_1(D,\omega)$ belong to the family in \S\ref{ex-1}, while the double extensions $\k_2(D,\omega)$ belong to the family in \S\ref{ex-2}.

\begin{prop}\label{dim6}
	Let $\g$ be a $6$-dimensional unimodular solvable Lie algebra. If $\g$ admits a Vaisman structure, then $\g$ is isomorphic to one of the following Lie algebras: 
	\begin{enumerate}[(i)]
	\item	$\R\times\h_5$ or $\R\ltimes_{D_r}\h_5$ for $r\in[-1,1]$, where $\h_5=\text{span}\{B,e_1,e_2,e_3,e_4\}$, with Lie brackets given by: $[e_1,e_2]=[e_3,e_4]=B$ and 
	\[  D_r=\left(\begin{array}{ccccc}      
	0    &    &              &      &     \\
	&  0 & -r            &      &     \\
	&  r & 0            &      &     \\
	&    &              &   0  & -1   \\
	&    &              &  1  &  0  \\
	\end{array}\right)\]
	in this basis. 
	
    \item $\R\times\mathfrak{s}_5$ or $\R\ltimes_{D_0}\mathfrak{s}_5$ where $\mathfrak{s}_5=\text{span}\{B,e_1,e_2,e_3,e_4\}$, with Lie brackets given by: $[e_1,e_3]=e_4$, $[e_1,e_4]=-e_3$, $[e_1,e_2]=[e_3,e_4]=B$ and $D_0$ is as in $(i)$ with $r=0$.
	\end{enumerate}
	Moreover, all these Lie algebras are pairwise non-isomorphic and the corresponding 
simply connected Lie groups admit lattices (with $r\in\mathbb{Q}$ for $\R\ltimes_{D_r}\mathfrak{h}_5$).
\end{prop}

\begin{proof}
It is easy to see, taking into account Lemma \ref{rescaling}, that $\k_1(D,\omega)$ is isomorphic to either $\R\times\h_5$ or $\R\ltimes_{D_r}\h_5$ with 
$r\in[-1,1]$. It can be verified that these Lie algebras are pairwise non-isomorphic. 
On the other hand, for $\k_2(D,\omega)$, using Theorem \ref{base linda} it can be easily seen that this Lie algebra is isomorphic to either $\R\times\mathfrak{s}_5$ or 
$\R\ltimes_{D_0}\mathfrak{s}_5$. 
Using Proposition \ref{g'-n} and after some computations, it can be seen that $\R\times\mathfrak{s}_5$ or $\R\ltimes_{D_0}\mathfrak{s}_5$ are not isomorphic to neither $\R\times\h_5$ nor 
$\R\ltimes_{D_r}\h_5$.

Note that $\R\ltimes_{D_r}\h_5$ corresponds to the Lie algebra $\g_{(r,1)}$ with $r\in[-1,1]$ studied in \S\ref{ex-1}. Therefore the existence of lattices for the group $G_{(r,1)}$ with $r\in \Q$ was established in Proposition \ref{lattices-oscillator}. In the case of $\R\times\mathfrak{s}_5$ and 
$\R\ltimes_{D_0}\mathfrak{s}_5$ the existence of lattices in the associated simply connected Lie group follows from \S\ref{ex-2} (see also \cite{AFV}).
\end{proof}

\medskip

It follows from Lemma \ref{rescaling} that the Lie algebra $\g_{(r,1)}=\R\ltimes_{D_r}\h_5,\, r\in\Q,$ from Proposition \ref{dim6} is isomorphic to $\g_{(a,b)}$ for 
some $a,b\in\Z$, $(a,b)=1$. Next we determine the first homology group and the first Betti number of the $6$-dimensional solvmanifolds 
$\Lambda_{k,j}\backslash G_{(a,b)}$ for $j=2\pi,\pi,\pi/2$, following \S \ref{ex-1}.

\medskip

$\bullet$ For the familiy $\Lambda_{k,2\pi}=2\pi\Z\times\Gamma_k$, the solvmanifold $\Lambda_{k,2\pi}\backslash G_{(a,b)}$ is diffeomorphic to a nilmanifold $S^1\times 
\Gamma_k\backslash H_5$, for all $(a,b)$. The first homology group of this nilmanifold is $\Z^5\oplus\Z_{2k}$ and hence its first Betti number is $b_1=5$.

\

$\bullet$ For the family $\Lambda_{k,\pi}$, the isomorphism class of this lattice depends only on the parity of the integers $a, b$. Then it 
is straightforward to verify that the first homology group of the solvmanifold $M:=\Lambda_{k,\pi}\backslash G_{(a,b)}$ and its first Betti number are given by:  
\begin{table}[h]
\begin{center}
\renewcommand{\arraystretch}{1.5}
\begin{tabular}{|c|c|c|}\hline
  $a,b$           &      $H_1(M,\Z)$                  & $b_1(M)$ \\ \hline
 $ab\equiv 1\, (\text{mod } 2)$ & $\Z\oplus\Z_{2k}\oplus \Z_2^4$    &   1      \\ \hline
 $ab\equiv 0\, (\text{mod } 2)$ & $\Z^3\oplus\Z_{2k}\oplus \Z_2^2$  &   3      \\ \hline
\end{tabular} 
\end{center}
\end{table}

\medskip

$\bullet$ For the family $\Lambda_{k,\pi/2}$, the isomorphism class of this lattice depends only on the congruence class of the integers $a, b$ modulo 4, and it can be seen that 
the first homology group of the solvmanifold $M:=\Lambda_{k,\frac{\pi}{2}}\backslash G_{(a,b)}$ and its first Betti number are given by: 
\begin{table}[h]
\begin{center}
\renewcommand{\arraystretch}{1.5}
\begin{tabular}{|c|c|c|}\hline
  $a,b$              &      $H_1(M,\Z)$                  & $b_1(M)$ \\ \hline
 $ab\equiv \pm1\, (\text{mod } 4)$ & $\Z\oplus\Z_{2k}\oplus \Z_2^2$    &   1      \\ \hline
 $ab\equiv 2\, (\text{mod } 4)$    & $\Z\oplus\Z_{2k}\oplus \Z_2^3$    &   1      \\ \hline
 $ab\equiv 0\, (\text{mod } 4)$    & $\Z^3\oplus\Z_{2k}\oplus \Z_2$    &   3      \\ \hline
\end{tabular} 
\end{center}
\end{table}

\

\section{Relation with other geometric structures}

\medskip

\subsection{CoK\"ahler Lie algebras}

In previous sections we have seen that any unimodular solvable Lie algebra admitting a Vaisman structure is the semidirect product of $\R$ with a Sasakian Lie algebra and this 
Sasakian Lie algebra, in turn, is a central extension of a K\"ahler flat Lie algebra. In what follows we will establish a relation with another type of almost contact metric Lie 
algebras, namely, coK\"ahler ones. In fact, we will show that a unimodular solvable Lie algebra admitting a Vaisman structure is a central extension of a coK\"ahler Lie 
algebra which is, in turn, a semidirect product of $\R$ with a K\"ahler flat Lie algebra.

\

Let $\g$ be a unimodular solvable Lie algebra that admits a Vaisman structure $(J,\pint)$, with $\omega$ its fundamental 2-form, $\theta$ the corresponding Lee form and $A\in\g$ 
its metric dual, as before. We know from Theorem \ref{Main Theorem-A} that $\g$ is a double extension $\g=\k(D,\omega')$ of the K\"ahler flat Lie algebra $\k$ by certain 
derivation $D$ of $\k_{\omega'}(\xi)$ such that $D':=D|_\k$ is a unitary derivation of $\k$. 

\begin{teo}\label{cokahler}
With notation as above, the Lie algebra $\mathfrak d=\R A\ltimes_{D'}\k$ admits a coK\"ahler structure $(\pint|_{\mathfrak d\times\mathfrak d},\phi, \xi, \eta)$, where 
$\phi\in\operatorname{End}({\mathfrak d})$ is defined by $\phi(aA+x)=Jx$ for $a\in\R,\,x\in\k$, and  $\eta:=\theta|_{\mathfrak d}$, $\xi:=A$. Moreover, if $\Phi$ denotes 
the (closed) fundamental  2-form on $\mathfrak d$ associated to this coK\"ahler structure, then $\g$ is isomorphic to the central extension $\mathfrak d_{\Phi}(JA)$.
\end{teo}

\begin{proof}
It is readily verified that $(\pint|_{\mathfrak d\times\mathfrak d},\phi, \xi, \eta)$ is an almost contact metric structure on $\mathfrak d$. Let us prove now that it is almost 
coK\"ahler, i.e., $d\eta=0$ and $d\Phi=0$, where $\Phi$ is the fundamental 2-form defined by $\Phi(x,y)=\la \phi x, y\ra$, $x,y\in\mathfrak d$. 

Since $[\mathfrak d,\mathfrak d]\subset \k=\ker\eta$, we have that $d\eta=0$. Now, for $a,b,c\in\R$ and $x,y,z\in\k$, we compute easily
\begin{align}\label{cosym} 
d\Phi(aA+x,bA+y,cA+z) & =d^\k \omega'(x,y,z)+a(\la D'y,Jz\ra-\la D'z,Jy\ra)\\ \nonumber
& \qquad +b(\la D'z,Jx\ra-\la D'x,Jz\ra)+c(\la D'x,Jy\ra-\la D'y,Jx\ra),
\end{align}
where $\omega'=\omega|_{\k\times \k}$. Since $\omega'$ is the K\"ahler form on $\k$, we have that $d^\k \omega'=0$. Since both $D$ and $J$ are skew-symmetric and 
$D'J|_\k=J|_\k D'$, we obtain that the last terms in \eqref{cosym} vanish, and therefore $d\Phi=0$.

To verify the normality of this structure, since $d\eta=0$ we only have to check that $N_\phi=0$. For $a,b\in\R$ and $x,y\in\k$, we compute
\[ N_\phi(aA+x,bA+y)=N^\k_{J|_\k}(x,y)-a(D'y+JD'Jy)+b(D'x+JD'Jx).\]
Since $N^\k_{J|_\k}=0$ and $D'J|_\k=J|_\k D'$, it follows that $N_\phi=0$.

To prove the last statement we compute the Lie bracket $[\cdot,\cdot]'$ on the central extension 
$\mathfrak d_{\Phi}(JA)$. We have that $JA$ is central and for $a,b\in\R$, $x,y\in\k$ we compute
\begin{align*}
[aA+x,bA+y]' & = \Phi(aA+x,bA+y)JA+[aA+x,bA+y]_{\mathfrak d} \\
& = \la \phi(aA+x),bA+y\ra JA + aD'y-bD'x+[x,y]_\k\\
& = \la Jx,y\ra JA + aDy-bDx+[x,y]_\k \\
& =\omega(x,y)JA+aDy-bDx+[x,y]_\k,
\end{align*}
which coincides with the Lie bracket on $\g$, according to Theorem \ref{JA-central} and \eqref{corchetek}. This completes the proof.
\end{proof}

\medskip

\begin{rem}\label{co-flat}
The first part of Theorem \ref{cokahler} follows also from \cite[Theorem 6.1]{FV}. Moreover, according to \cite[Proposition 6.4]{FV}, $(\mathfrak d,\pint)$ is a flat Lie algebra, 
since $\mathfrak d$ is unimodular. If, as above, $\k=\z\oplus\h\oplus\mathfrak k'$ is the decomposition of $\k$ given by Proposition \ref{gplana}, then the corresponding 
decomposition of $\mathfrak d$ is $\mathfrak{d}=\widetilde{\z}\oplus\widetilde{\h}\oplus\mathfrak{d}'$, where $\widetilde{\z}=\ker(D'|_{\z\cap J\z})$, $\widetilde{\h}=\R A\oplus 
\h$ and $\mathfrak{d}'=\k'\oplus \operatorname{Im}(D'|_{\z\cap J\z})$, whenever $D'\neq 0$. If $D'=0$, we have that $\widetilde{\z}=\R A\oplus \z$, $\widetilde{\h}=\h$ and 
$\mathfrak{d}'=\k'$.
\end{rem}

\

\subsection{Left-symmetric algebra structures}
In this section we show that any unimodular solvable Lie algebra equipped with either a Sasakian or a Vaisman structure admits also another kind of algebraic structure with a  
geometrical interpretation. 

A \textit{left-symmetric algebra (LSA)} structure on a Lie algebra $\mathfrak a$ is a bilinear product $\mathfrak a\times\mathfrak a \longrightarrow\mathfrak a,\,(x,y)\mapsto 
x\cdot y$, which satisfies
\begin{equation}\label{tfree} 
[x,y]=x\cdot y-y\cdot x
\end{equation}
and 
\begin{equation}\label{flat}
x\cdot(y\cdot z)-(x\cdot y)\cdot z=y\cdot(x\cdot z)-(y\cdot x)\cdot z,
\end{equation}
for any $x,y,z\in\mathfrak a$. See \cite{Bur} for a very interesting review on this subject.

LSA structures have the following well known geometrical interpretation. If $G$ is a Lie group and $\g$ is its Lie algebra, then LSA structures on $\g$ are in one-to-one 
correspondence with left invariant flat torsion-free connections $\nabla$ on $G$. Indeed, this correspondence is given as follows: $\nabla_xy=x\cdot y$ for $x,y\in\g$. Since this 
connection is left invariant, any quotient $\Gamma\backslash G$ of $G$ by a discrete subgroup $\Gamma$ also inherits a flat torsion-free connection. It is well known that one can 
study the completeness of the connection $\nabla$ on $G$ in terms only of the corresponding LSA structure on $\g$. Indeed, $\nabla$ is geodesically complete if and only if the 
endomorphisms $\operatorname{Id}+\rho(x)$ of $\g$ are bijective for all $x\in\g$, where $\rho(x)$ denotes right-multiplication by $x$, i.e.,  $\rho(x)y=y\cdot x$ (see for instance 
\cite{Se}). In this case, it is said that the LSA structure is \textit{complete}.

Note that the Levi-Civita connection of a flat metric on a Lie algebra is an example of a complete LSA structure, since the corresponding left invariant metric on any associated 
Lie group is homogeneous and therefore complete. 

\

\begin{teo}\label{LSA}
Let $(\h,\pint)$ be a flat Lie algebra and let $\beta$ denote a $2$-form which is parallel with respect to the Levi-Civita connection $\nabla$ of $\pint$ (hence $\beta$ is 
closed). 
Then the central extension $\g=\h_\beta(\xi)$ admits an LSA structure defined by
\[ (a\xi+x)\cdot(b\xi+y)=\frac12 \beta(x,y)\xi+\nabla_xy, \quad a,b\in\R,\, x,y\in\h.\]
Furthermore, this LSA structure is complete.
\end{teo}

\begin{proof}
 Taking into account that $\nabla$ is a torsion-free connection on $\h$ and the fact that $\xi$ is central in $\g$, it is easily verified that 
 \[ (a\xi+x)\cdot(b\xi+y)-(b\xi+y)\cdot(a\xi+x)=[a\xi+x,b\xi+y]_\beta.  \]
Therefore, \eqref{tfree} holds for this product. In order to prove \eqref{flat}, let us compute
 \begin{align*}
 (a\xi+x)\cdot((b\xi+y)\cdot(c\xi+z)) & -(b\xi+y)\cdot((a\xi+x)\cdot(c\xi+z))  = \\ & = \frac12\beta(x,\nabla_yz)\xi+\nabla_x\nabla_yz - \frac12\beta(y,\nabla_xz)\xi -\nabla_y\nabla_xz 
\\
 & = -\frac12\beta(\nabla_yx,z)\xi + \frac12\beta(\nabla_xy,z)\xi+\nabla_x\nabla_yz -\nabla_y\nabla_xz\\
 & = \frac12\beta([x,y],z)\xi+\nabla_{[x,y]}z \\
 & =[a\xi+x,b\xi+y]_\beta\cdot (c\xi+z),
 \end{align*}
where we have used $\nabla\beta=0$ in the third line and the fact that $\nabla$ is torsion-free and flat in the fourth line. This is equivalent to \eqref{flat}, thus this product 
is an LSA structure on $\g$.

Let us prove the completeness. Fix $b\xi+y\in \g$ and assume that $(\operatorname{Id}+\rho(b\xi+y))(a\xi+x)=0$. Then
\begin{align*} 
0 & = a\xi +x + \frac12 \beta(x,y)\xi+\nabla_xy \\
& = \left(a+\frac12 \beta(x,y)\right)\xi + \left(x+\nabla_xy\right),
\end{align*}
hence $a+\frac12 \beta(x,y)=0$ and $x+\nabla_xy=0$. But, since $\nabla$ itself is a complete LSA structure on $\h$, it follows that $x=0$. This implies that $a=0$, and the 
completeness follows.
\end{proof}

\medskip

\begin{cor}\label{LSA2}
 Let $\g$ be a unimodular solvable Lie algebra. 
 \begin{enumerate}
  \item If $\g$ carries a Vaisman structure, then $\g$ admits a complete LSA structure.
  \item If $\g$ carries a Sasakian structure, then $\g$ admits a complete LSA structure.
 \end{enumerate}
\end{cor}

\begin{proof}
 If $\g$ admits a Vaisman structure then, according to Theorem \ref{cokahler}, $\g$ is a central extension of a coK\"ahler Lie algebra $\mathfrak d$ by the fundamental $2$-form 
$\Phi$, which is parallel. As mentioned in Remark \ref{co-flat}, $\mathfrak d$ is flat, and therefore (1) follows from Theorem \ref{LSA}.
 
 If $\g$ admits a Sasakian structure, then it follows from Corollary \ref{sasaki-centro2} that $\g$ is a central extension of a K\"ahler flat Lie algebra. Hence, (2) follows from 
Theorem \ref{LSA} again.
\end{proof}

\smallskip

%Any LSA structure on a Lie algebra determines a unique left invariant flat torsion-free connection on any associated Lie group, which also induces a flat torsion-free connection 
%on any quotient by a discrete subgroup. Therefore, Corollary \ref{LSA2} implies the following result:

\begin{cor}
Any solvmanifold admitting either an invariant Vaisman structure or an invariant Sasakian structure carries a geodesically complete flat torsion-free connection.
\end{cor}

\

\end{document}